\RequirePackage[l2tabu, orthodox]{nag}
\documentclass[final,11pt, a4paper]{article}

\newcommand{\mytitle}{Approximate Euclidean Steiner Trees}

\usepackage[a4paper,twoside,
inner=2.54cm,outer=2.54cm,top=2.54cm,bottom=2.54cm,headsep=0.0cm
]{geometry}%

\usepackage[table]{xcolor}
\usepackage{upgreek}
\usepackage{tikz,bm}
\usetikzlibrary{arrows}

\usepackage{yhmath}
\usepackage[final]{microtype}
\makeatletter
\def\MT@register@subst@font{\MT@exp@one@n\MT@in@clist\font@name\MT@font@list
   \ifMT@inlist@\else\xdef\MT@font@list{\MT@font@list\font@name,}\fi}
\makeatother 

\usepackage[small]{titlesec}

\usepackage{amsmath, amssymb, amsfonts, amsthm, mathtools, enumerate}
\allowdisplaybreaks[1]
\usepackage{fixmath}
\mathtoolsset{centercolon}

\usepackage[latin1]{inputenc}
\usepackage[T1]{fontenc}

\usepackage[draft=false,pdftex]{hyperref}
\hypersetup{
    bookmarks=true,         
    unicode=false,          
    pdftoolbar=true,        
    pdfmenubar=true,        
    pdffitwindow=false,     
    pdfstartview={FitH},    
    pdftitle={\mytitle},    
    pdfauthor={Charl Ras, Konrad Swanepoel and Doreen Thomas},     
    pdfkeywords={approximate Steiner tree, Euclidean Steiner problem, shortest networks, approximation}, 
    pdfnewwindow=false,      
    colorlinks=true,       
    linkcolor=black,          
    citecolor=black,        
    filecolor=black,      
    urlcolor=black           
}

\theoremstyle{plain}
\newtheorem{theorem}{Theorem}[section]
\newtheorem{lemma}[theorem]{Lemma}
\newtheorem{corollary}[theorem]{Corollary}
\newtheorem{conjecture}[theorem]{Conjecture}
\newtheorem{proposition}[theorem]{Proposition}

\theoremstyle{definition}

\newcommand{\setbuilder}[2]{\left\{#1\colon#2\right\}}
\newcommand{\set}[1]{\left\{#1\right\}}

\newcommand{\epsi}{\varepsilon}

\renewcommand{\leq}{\leqslant}
\renewcommand{\geq}{\geqslant}
\newcommand{\norm}[1]{\left\lVert#1\right\rVert}

\newcommand{\abs}[1]{\left\lvert#1\right\rvert}
\newcommand{\conj}[1]{\overline{#1}}

\newcommand{\myangle}{\sphericalangle}
\newcommand{\Ray}[1]{\overrightarrow{#1}}
\newcommand{\arc}[1]{\wideparen{#1}}
\newcommand{\length}[1]{\lvert#1\rvert}

\newcommand{\numbersystem}[1]{\mathbb{#1}}

\newcommand{\bC}{\numbersystem{C}}

\newcommand{\bN}{\numbersystem{N}}

\newcommand{\bR}{\numbersystem{R}}

\newcommand{\bZ}{\numbersystem{Z}}

\newcommand{\collection}[1]{{\mathcal#1}}
\newcommand{\CA}{\collection{A}}

\newcommand{\CT}{\collection{T}}

\DeclareMathOperator{\diam}{diam}

\newcommand{\define}[1]{%
    \emph{#1}%
}

\title{\mytitle}
\author{Charl Ras\thanks{School of Mathematics and Statistics, University of Melbourne. 
\href{mailto:cjras@unimelb.edu.au}{\texttt{cjras@unimelb.edu.au}}},\; Konrad J.\ Swanepoel\thanks{Corresponding author. Department of Mathematics, London School of Economics and Political Science
. 
    \href{mailto:k.swanepoel@lse.ac.uk}{\texttt{k.swanepoel@lse.ac.uk}}.
},\; Doreen Thomas\thanks{Department of Mechanical Engineering, University of Melbourne. 
\href{mailto:doreen.thomas@unimelb.edu.au}{\texttt{doreen.thomas@unimelb.edu.au}}.}}
\date{}

\begin{document}
\maketitle

\begin{abstract}
An approximate Steiner tree is a Steiner tree on a given set of terminals in Euclidean space such that the angles at the Steiner points are within a specified error $e$ from 120 degrees.
This notion arises in numerical approximations of minimum Steiner trees (W.~D.~Smith, Algorithmica, \textbf{7} (1992), 137--177).
We investigate the worst-case relative error of the length of an approximate Steiner tree compared to the shortest tree with the same topology.
Rubinstein, Weng and Wormald (J. Global Optim.\ \textbf{35} (2006), 573--592) conjectured that this relative error is at most linear in $e$, independent of the number of terminals.
We verify their conjecture for the two-dimensional case as long as the error $e$ is sufficiently small in terms of the number of terminals.
We derive a lower bound linear in $e$ for the relative error in the two-dimensional case when $e$ is sufficiently small in terms of the number of terminals.
We find improved estimates of the relative error for larger values of $e$, and calculate exact values in the plane for three and four terminals.


\end{abstract}

\section{Introduction}
The Euclidean Steiner problem asks for a tree of shortest total length that interconnects a given collection of points or \emph{terminals} in Euclidean space.
For example, to interconnect the four vertices of a square in the plane, a shortest tree contains two further points apart from the four terminals (Fig.~\ref{fig:square}).
\begin{figure}[t]
\centering
\definecolor{ffqqqq}{rgb}{1.,0.,0.}
\begin{tikzpicture}[line cap=round,line join=round,>=triangle 45,x=1.0cm,y=1.0cm, scale=0.6]
\clip(1.15553890927703,3.050225419864017) rectangle (6.072027371278206,7.995634402229929);
\draw[gray] (1.6,3.5)-- (5.64,3.5);
\draw[gray] (5.64,3.5)-- (5.64,7.54);
\draw[gray] (5.64,7.54)-- (1.6,7.54);
\draw[gray] (1.6,7.54)-- (1.6,3.5);
\draw [line width=1.2pt,color=ffqqqq] (1.6,7.54)-- (2.7662475437630474,5.52);
\draw [line width=1.2pt,color=ffqqqq] (2.7662475437630474,5.52)-- (1.6,3.5);
\draw [line width=1.2pt,color=ffqqqq] (2.7662475437630474,5.52)-- (4.473752456236958,5.52);
\draw [line width=1.2pt,color=ffqqqq] (4.473752456236958,5.52)-- (5.64,7.54);
\draw [line width=1.2pt,color=ffqqqq] (4.473752456236958,5.52)-- (5.64,3.5);
\begin{scriptsize}
\draw [fill=black] (1.6,3.5) circle (1.5pt);
\draw [fill=black] (5.64,3.5) circle (1.5pt);
\draw [fill=black] (5.64,7.54) circle (1.5pt);
\draw [fill=black] (1.6,7.54) circle (1.5pt);
\draw [fill=ffqqqq] (2.7662475437630474,5.52) circle (2.5pt);
\draw [fill=ffqqqq] (4.473752456236958,5.52) circle (2.5pt);
\end{scriptsize}
\end{tikzpicture}
\caption{Minimum Steiner tree (in red) of the vertices of a square}\label{fig:square}
\end{figure}
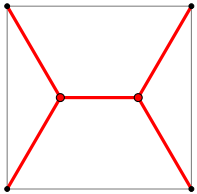
Such a shortest tree is called a \define{minimum Steiner tree} on the given collection of terminals, and the additional points are called \emph{Steiner points}.
The Steiner problem is well studied, especially in the plane.
An overview of the extensive literature on this problem can be found in the monographs of Hwang, Richards and Winter \cite{HRW}, Cieslik \cite{Cieslik}, Pr\"omel and Steger \cite{PS}, and the recent Brazil and Zachariasen \cite{BZ}.
For more on the history of the problem, see Boltyanski, Martini, and Soltan \cite{BMS} and the recent Brazil, Graham, Thomas, and Zachariasen \cite{BGTZ}.

It is well known that a minimum Steiner tree in Euclidean space has maximum degree three, that the Steiner points always have degree three, and that each angle spanned by two edges with a common endpoint is at least $120$ degrees, and exactly $120$ degrees at each Steiner point \cite[Section~6.1]{HRW}.
In the plane, there is a ruler-and-compass construction of a minimum Steiner tree once the graph structure (or \emph{topology}) is known.
This construction, also known as the Melzak algorithm \cite{Melzak}, can be done in linear time \cite{Hwang}.
On the other hand, determining the topology of a minimum Steiner tree is hard.
There is a super-exponential number of different topologies \cite{GP}, and it is already NP-hard to decide whether a given set of points in the plane has a Steiner tree of length smaller than a given length \cite{GGJ}.
On the other hand, the GeoSteiner package of Warme, Winter and Zachariasen quickly finds minimum Steiner trees on a relatively large number of points in the plane \cite{JWWZ}.

There are polynomial-time approximation schemes to calculate minimum Steiner trees in Euclidean space (Arora \cite{A} and Mitchell \cite{M}; see also \cite{RS}).
However, for the actual implementation of these schemes, there has been progress so far only for certain planar problems \cite{TM}.
A major obstacle in the implementation of these schemes for higher-dimensional problems is that their time complexity depends doubly exponentially on the dimension, and there is some evidence that this is unavoidable \cite{Trevisan}.

In higher dimensions, the Steiner points are not necessarily constructible anymore, and finding the optimal Steiner points results in solving high-degree algebraic equations, or solving a convex optimisation problem numerically \cite{Smith}.
See the papers \cite{BTW, FBWZ, GP, vLA, RTW, RWW, ST, Smith} for work on finding minimum Steiner trees in Euclidean spaces of dimension at least $3$.
We mention that Steiner trees in $3$-space have been considered in theoretical investigations of multiquarks in particle physics \cite{BC}, and in higher dimensions have been used to determine phylogenetic trees \cite{BTNWWZ}.

One problem arising from a numerical approach is that of estimating how close an approximation is to a locally minimum Steiner tree with a given Steiner topology.
Rubinstein, Weng and Wormald \cite{RWW} studied the relative error in the length of an approximate Steiner tree in terms of how far the angles at Steiner points deviate from $120$ degrees.
This paper is a further contribution to this topic.

Before we can give an exact definition of the relative error, we have to introduce our terminology and notation in the next section.
Then in Section~\ref{section:problem} we define the relative error and formulate the main conjectures from \cite{RWW}.
Our results are stated and summarized in Section~\ref{section:results}.
Section~\ref{section:monotone} is a brief discussion of the monotonicity of the relative error as the number of terminals increases.
In Section~\ref{section:large}, we prove our results for large relative errors.
For small relative errors, we subdivide the proofs into a section on upper bounds (Section~\ref{section:small}) and lower bounds (Section~\ref{section:construction}).
We conclude in Section~\ref{section:conclusion} with some remarks.
There are two tedious induction proofs of results in Section~\ref{section:construction} that are presented in an Appendix.

\section{Terminology}
We define a \define{Steiner topology for $n$ terminals} to be a tree $\CT$ with $n$ special vertices $t_1,\dots,t_n$, called \define{terminals}, all of degree at most $3$, and all other vertices, called \define{Steiner points}, of degree exactly $3$.
A Steiner topology is \define{full} if all terminals have degree $1$.
Let $N=\{p_1,\dots,p_n\}$ be a family of $n$ points in $\bR^d$ (allowing repeated points).
A \define{Steiner tree} $T$ with topology $\CT$ for $N$ is a representation of $\CT$ in $\bR^d$, with each $t_i$ represented by $p_i$, each Steiner point of $\CT$ represented by an arbitrary point of $\bR^d$, and edges represented by straight-line segments.
We say that such a Steiner tree \define{interconnects} $N$.
A Steiner tree is \emph{full} if its topology is full.
We allow Steiner points to coincide with each other and with terminals, hence for edges incident to a Steiner point to be of length $0$.
An edge of length $0$ is called \define{degenerate}, and we say that a Steiner tree that contains a degenerate edge is \define{degenerate}.
We allow edges to intersect each other.

The (convex) angle determined by two edges $xy$ and $xz$ with a common endpoint $x$ is denoted $\myangle yxz$.
Its angular measure is also denoted by $\myangle yxz$, and we assume that angular measures are in the interval $[0,\pi]$.
We use radians for angular measure throughout the paper, except in a few places where it will be clear that we use degrees.

We denote the Euclidean length of an edge $pq$ by $\length{pq}$.
The \define{length} $L(T)$ of a tree $T$ is the sum of the Euclidean lengths of its edges.
Among all the trees that interconnect a given set $N$ of terminals there is at least one tree of minimum length, which we call a \define{minimum Steiner tree} of $N$.
We define a \define{locally minimum Steiner tree} to be a non-degenerate tree with a Steiner topology and with all angles spanned by the edges at each vertex at least~$2\pi/3$.
Since each Steiner point in a Steiner topology has degree $3$, it easily follows (in any dimension) that each of the three angles at a Steiner point is exactly $2\pi/3$ and that the three edges incident to the Steiner point are coplanar.
As mentioned above, any minimum Steiner tree is a locally minimum Steiner tree.
A \define{full minimum Steiner tree} is a minimum Steiner tree that is also full.

We denote the largest integer not greater than $x$ by $\lfloor x\rfloor$.

\section{Formulation of the Problem, Conjectures and Previous Results}\label{section:problem}
In \cite{RWW} the following notions were introduced.
Let $\epsi\geq 0$ be given.
An \define{$\epsi$-approximate Steiner tree} is a tree with a Steiner topology, with all the angles spanned by the edges at each Steiner point belonging to the interval $[2\pi/3-\epsi,2\pi/3+\epsi]$.
Note that a $0$-approximate Steiner tree is the same as a locally minimum Steiner tree.
(In \cite{RWW} the distinction was made between a \emph{pseudo-Steiner point} of an $\epsi$-approximate Steiner tree and a \emph{Steiner point} of a locally minimum Steiner tree.
For the sake of simplicity we make no such distinction and use the term \emph{Steiner point} for both.)

For $d\geq 2$, $n\geq 3$ and $\epsi\geq 0$, let $\CA_\epsi^d(n)$ denote the set of all full $\epsi$-approximate Steiner trees on $n$ terminals in~$\bR^d$, and let $\overline{\CA}_\epsi^d(n)$ denote the subset of all $T\in\CA_\epsi^d(n)$ for which the terminals have a minimum Steiner tree with the same topology as $T$.
In particular, $\CA_0^d(n)$ is the set of all full locally minimum Steiner trees on $n$ terminals in~$\bR^d$, and $\overline{\CA}_0^d(n)$ is the set of all full minimum Steiner trees on $n$ terminals in $\bR^d$.

Given a tree $T$ in $\bR^d$ with Steiner topology $\CT$, let $S(T)$ denote the shortest tree in $\bR^d$ on the terminals of $T$ with topology $\CT$, where we allow degenerate shortest trees.
Even though $S(T)$ is not necessarily a Steiner tree (see for instance \cite[Figure~1.7]{BZ}), it can be shown that $S(T)$ is always unique \cite[Section~4]{GP}.

Rubinstein, Weng and Wormald \cite{RWW} defined the following two quantities:
\[ F_d(\epsi,n) = \sup\setbuilder{\frac{L(T)-L(S(T))}{(L(S(T))}}{T\in\CA_\epsi^d(n)}\]
and
\[ \overline{F}_d(\epsi,n) = \sup\setbuilder{\frac{L(T)-L(S(T))}{(L(S(T))}}{T\in\overline{\CA}_\epsi^d(n)},\]
and made the following conjectures in the case $d\geq 3$.
Although they did not consider the $2$-dimensional case, we include it, as it is also still open, and most of our results will be in the plane.
\begin{conjecture}\label{conj1}
For any $d\geq 2$ there exist $\epsi_0>0$ and $C_d>0$ such that for all $\epsi\in(0,\epsi_0)$ and $n\in\bN$, $F_d(\epsi,n) < C_d\epsi$.
\end{conjecture}
\begin{conjecture}
For any $d\geq 2$ there exist $\epsi_0>0$ and $C_d>0$ such that for all $\epsi\in(0,\epsi_0)$ and $n\in\bN$, $\overline{F}_d(\epsi,n) < C_d\epsi$.
\end{conjecture}
The second conjecture is weaker than the first, but it seems difficult to deduce an upper bound for $\overline{F}_d$ that cannot already be deduced for $F_d$.
Rubinstein, Weng and Wormald \cite{RWW} showed that for $\epsi<1/n^2$, $F_d(\epsi,n)\leq C_d (\epsi\log n + \epsi^2 n^3)$.
They also consider larger values of~$\epsi$.

\section{Overview of New Results}\label{section:results}
Our results are summarized in Table~\ref{table1}.
\begin{table}\def\arraystretch{1}
\centering
\setlength\minrowclearance{7pt}
\begin{tabular}{|l|l|l|} \hline
\qquad Range & \qquad Bound & \\[1mm] \hline\hline
$\epsi=O(1/n^2)$ & $F_2(\epsi,n) = O(\epsi)$ & Theorem~\ref{thm:plane}\\[2mm] \hline
$\epsi < \dfrac{\pi}{n-2}$ & $F_2(\epsi,n) \leq \dfrac{1}{\cos \tfrac{(n-2)\epsi}{2}} - 1 =O(n^2\epsi^2)$ & Theorem~\ref{thm:plane}\\[5mm] \hline
$\epsi < \dfrac{1}{(\log n)^2}$ & $F_2(\epsi,n)\geq G_2(\epsi,n) = \upOmega((\log n)^2\epsi^2)$ & Theorem~\ref{theorem:lowerbound}\\[5mm] \hline
$\epsi\leq\pi/6$ & $F_2(\epsi,n)\leq 2n-4$ & Proposition~\ref{plane} \\[2mm] \hline
$0 < \epsi < 2\pi/3$ & $F_d(\epsi,n) = O\left(\left(\dfrac{\cos(\epsi/2)}{\sin(\pi/3-\epsi/2)}\right)^n\right)$ & Theorem~\ref{upperbound} \\[5mm] \hline
$\epsi=\pi/3$ & $F_2(\epsi,n)\geq G_2(\epsi,n) = \upOmega(\log n)$ & Theorem~\ref{polynomial} \\[2mm] \hline
$\pi/3<\epsi<2\pi/3$ & $F_2(\epsi,n)\geq G_2(\epsi,n) = \upOmega(n^{c(\epsi)})$ & \\
& \qquad where $0< c(\epsi)\nearrow\infty$ as $\epsi\to2\pi/3$ & Theorem~\ref{polynomial} \\[2mm] \hline
$0<\epsi<\pi/3$ & $F_2(\epsi,3)=G_2(\epsi,3)=\dfrac{1}{\cos(\epsi/2)} - 1$ & Proposition~\ref{prop:3unfold}\\ 
& $F_2(\epsi,4)=G_2(\epsi,4) = \dfrac{1}{\cos\epsi} - 1$ & Proposition~\ref{prop:4unfold} \\[4mm] \hline
\end{tabular}
\caption{Summary of results}\label{table1}
\end{table}
Our first main result is an upper bound for the relative error in the plane.
\begin{theorem}\label{thm:plane}
If $n\geq 3$ and $0<\epsi<\pi/(n-2)$, then \[F_2(\epsi,n)\leq \frac{1}{\cos \tfrac{(n-2)\epsi}{2}}-1.\]
\end{theorem}
The proof is in Section~\ref{section:small}.
As a consequence, Conjecture~\ref{conj1} holds in the plane if $\epsi$ is sufficiently small, depending on $n$.
\begin{corollary}\label{cor:plane}
If $0<\epsi<\pi/(n-2)$, then $F_2(\epsi,n) = O(n^2\epsi^2)$.
Consequently, if $\epsi = O(1/n^2)$ as $n\to\infty$, then $F_2(\epsi,n)=O(\epsi)$.
\end{corollary}
In \cite{RWW} an example is given that shows that Conjecture~\ref{conj1} is sharp for each $d\geq 3$.
Our second main result, Theorem~\ref{theorem:lowerbound}, is a lower bound for $F_2$ that shows that Conjecture~\ref{conj1} is already sharp in the plane for sufficiently small $\epsi$.
\begin{theorem}\label{theorem:lowerbound}
For any $k\geq 1$, if $\epsi=c/k^2$ with $0<c<1$, then $F_2(\epsi,2^k+1) > \frac{c}{24}\epsi$.
Consequently, if $\epsi<(\log_2 n)^{-2}$, then $F_2(\epsi,n) =\upOmega((\log n)^2\epsi^2)$.
\end{theorem}
The proof is in Section~\ref{section:construction}.
In Section~\ref{section:large}, we show some bounds for larger $\epsi$.

In the above definition of $F_d$, we consider the worst-case relative error between a full $\epsi$-approximate Steiner tree $T$ on $n$ terminals and the shortest tree $S(T)$ with the same topology as $T$, even though $S(T)$ may have a degenerate topology.
Instead, we could restrict ourselves to trees $T$ for which $S(T)$ is non-degenerate.
Note that for any $T\in\CA_\epsi^d(n)$, $S(T)$ is non-degenerate iff $S(T)$ is a locally minimum Steiner tree.
We therefore introduce the following variants of the previous two quantities:
\[ G_d(\epsi,n) = \sup\setbuilder{\frac{L(T)-L(S(T))}{(L(S(T))}}{T\in\CA_\epsi^d(n), S(T)\in\CA_0^d(n)}\]
and
\[ \overline{G}_d(\epsi,n) = \sup\setbuilder{\frac{L(T)-L(S(T))}{(L(S(T))}}{T\in\overline{\CA}_\epsi^d(n), S(T)\in\overline{\CA}_0^d(n)}\]
Clearly, $G_d(\epsi,n)\leq F_d(\epsi,n)$ and $\overline{G}_d(\epsi,n)=\overline{F}_d(\epsi,n)$.
The construction that we make to prove the lower bounds of Theorem~\ref{theorem:lowerbound} in fact gives a lower bound for $G_2(\epsi,n)$ for certain values of $n$, as in the next theorem.
\begin{theorem}\label{theorem:lowerbound'}
For any $k\geq 1$, if $\epsi=c/k^2$ with $0<c<1$, then $G_2(\epsi,2^k+1) > \frac{c}{24}\epsi$.
\end{theorem}
Unfortunately we do not know whether $G_2(\epsi,n)$ is monotone in $n$ (see 
the next section), so we cannot state a lower bound for general $n$.

\section{\texorpdfstring{Monotonicity of $F_d$ and $G_d$}{Monotonicity of F d and G d}}\label{section:monotone}
In many of the examples constructed in this paper, the number of terminals is of a special form such as a power of $2$.
In order to make general statements for all $n$, we need to know that $F_d$ and $G_d$ are monotone in $n$.
Monotonicity in $\epsi$ and in $d$ are straightforward.
Indeed, if $0\leq \epsi_1 < \epsi_2$, then an $\epsi_1$-approximate Steiner tree is also an $\epsi_2$-approximate Steiner tree, hence $F_d(\epsi_1,n)\leq F_d(\epsi_2,n)$, $G_d(\epsi_1,n)\leq G_d(\epsi_2,n)$ and $\overline{F}_d(\epsi_1,n)\leq \overline{F}_d(\epsi_2,n)$.
It is also clear that $F_d$, $G_d$ and $\overline{F}_d$ are monotone in $d$:
\[ F_2\leq F_3\leq\dots,\quad G_2\leq G_3\leq\dots\quad\text{and}\quad \overline{F}_2\leq\overline{F}_3\leq\dots.\]

It is still relatively simple to show that $F_d$ is also monotone in $n$, as we show next.
\begin{proposition}\label{monotone}
For any $d\geq 2$, $\epsi>0$ and $n\geq 3$, $F_d(\epsi,n)\leq F_d(\epsi,n+1)$.
\end{proposition}
\begin{proof}
Consider any $\epsi$-approximate Steiner tree $T$ with a full Steiner topology on $n$ terminals.
Let $S$ be a shortest tree with the same terminals set and with the same (possibly degenerate) topology as $T$.
We show that
\begin{equation}\label{ineq}
F_d(\epsi,n+1)\geq L(T)/L(S) -1.
\end{equation}
Let $\delta>0$ be arbitrary.
Modify $T$ to obtain an $\epsi$-approximate Steiner tree $T'$ on $n+1$ terminals as follows.
Choose any terminal $t$ of $T$.
It is joined to a Steiner point $s$ of $T$.
Let $t_1$ and $t_2$ be two points at distance $\delta$ from $t$ such that the three angles at $t$ are equal: $\myangle t_1tt_2=\myangle t_1ts=\myangle t_2ts$.
(Thus, $t_1$, $t_2$, $t$ and $s$ have to be coplanar.)
If we consider $t_1$ and $t_2$ to be two new terminals, and consider $t$ to be a Steiner point, then we obtain an $\epsi$-approximate Steiner tree $T'$ on $n+1$ terminals of length $L(T')=L(T)+2\delta$.

We modify $S$ by adding the edges $t_1t$ and $t_2t$ to obtain a tree $S'$ with the same topology as $T'$ (allowing degenerate topologies).
Then $L(S(T'))\leq L(S')=L(S)+2\delta$, and
\[ F_d(\epsi,n+1)\geq\frac{L(T')}{L(S(T'))}-1\geq\frac{L(T)+2\delta}{L(S)+2\delta}-1.\]
Since this holds for all $\delta>0$, \eqref{ineq} follows.
Since \eqref{ineq} holds for an arbitrary $\epsi$-approximate Steiner tree on $n$ terminals,
\[F_d(\epsi,n+1)\geq\sup L(T)/L(S)-1=F_d(\epsi,n). \qedhere\]
\end{proof}
The monotonicity of $G_d(\epsi,n)$ in $n$ seems to be subtler, and we have only been able to show it for~$d\geq 3$.
\begin{proposition}\label{monotoneG}
For any $d\geq 3$, $\epsi>0$ and $n\geq 3$, $G_d(\epsi,n)\leq G_d(\epsi,n+1)$.
\end{proposition}
\begin{proof}
Let $\delta>0$ be arbitrary.
Let $T$ be a full $\epsi$-approximate Steiner tree on $n$ terminals in $\bR^d$, such that $S(T)$ is non-degenerate (in particular, $S(T)$ is still full).
Choose any terminal $t$ of $T$.
It is joined to a Steiner point $s$ in $T$ and also to a Steiner point $s'$ in $S(T)$.
Choose a point $t_1$ such that $t_1t$ is perpendicular to $ts$ and to $ts'$, and $\length{tt_1}=\sqrt{3}\delta$.
Let $t_2$ be the unique point such that $t$ is the midpoint of $t_1t_2$.
Without loss of generality, $\delta<\length{ts}, \length{ts'}$.
Then there exists a unique point $s_2$ on $st$ such that $\myangle t_1s_2t_2=2\pi/3$ and a unique point $s'_2$ on $s't$ such that $\myangle t_1s'_2t_2=2\pi/3$.
Let $T'$ be the tree obtained from $T$ by removing $t$ and $st$, and adding the Steiner point $s_2$, terminals $t_1$ and $t_2$, and edges $ss_2$, $t_1s_2$ and $t_2s_2$.
Then $T'$ is an $\epsi$-approximate Steiner tree on $n+1$ terminals, and $L(T')=L(T)+3\delta$.
Also, $S(T')$ is the tree obtained from $S(T)$ by removing $t$ and $s't$, and adding the Steiner point $s'_2$, terminals $t_1$ and $t_2$, and edges $s's'_2$, $t_1s'_2$ and $t_2s'_2$.
Then $L(S(T'))=L(S(T))+3\delta$.
We conclude that \[ G_d(\epsi,n+1) \geq\frac{L(T')}{L(S(T'))}-1 = \frac{L(T)+3\delta}{L(S(T))+3\delta}-1,\]
and by letting $\delta\to 0$ and taking the $\sup$ of the right-hand side, the proof is finished.
\end{proof}

We have not been able to show that $\overline{F}_d(\epsi,n)=\overline{G}_d(\epsi,n)$ is monotone in $n$.
We are also not sure whether $G_2(\epsi,n)\leq G_2(\epsi,n+1)$ or $\overline{F}_2(\epsi,n)\leq\overline{F}_2(\epsi,n+1)$ always hold.

\section{\texorpdfstring{Results for Large $\epsi$}{Results for Large Epsilon}}\label{section:large}
This section contains upper and lower bounds for $F_d$ for values of $\epsi$ that are independent of $n$.
In Proposition~\ref{plane} we obtain the modest upper bound of $2n-4$ for $F_2(\epsi,n)$, as long as $\epsi\leq\pi/6$.
We do not know of any better upper bound in the plane for small and fixed~$\epsi$.
In Theorem~\ref{upperbound} we give an explicit upper bound for $F_d(\epsi,n)$ for all values of $\epsi<2\pi/3$.
For instance, we obtain $F_d(\epsi,n)\leq O\left(\left(2/\sqrt{3}+\epsi\right)^n\right)$ for small $\epsi$.

Theorem~\ref{polynomial} sharpens Lemma~2.2 of \cite{RWW} in the range $\epsi\in(\pi/3,2\pi/3)$ by giving a lower bound for $F_d$ for all $d\geq 2$ of the form $n^{\alpha(\epsi)}$ where $\alpha(\epsi)$ is an explicit function of~$\epsi$.
In particular, it will follow that if $\epsi>105.6\dots^\circ$, then $\alpha(\epsi)>2$, hence the lower bound grows superquadratically.
This indicates that Theorem~2.1 of \cite{RWW} can only hold if $\epsi$ is sufficiently small.
We also obtain a lower bound for $\epsi=\pi/3$ of the form $\upOmega(\log n)$.

\begin{proposition}\label{plane}
If $\epsi\leq\pi/6$ and $n\geq 3$ then $F_2(\epsi,n)\leq 2n-4$.
\end{proposition}
\begin{proof}
Since $2\pi/3-\epsi\geq\pi/2$, it follows that each Steiner point of an $\epsi$-approximate Steiner tree $T$ is in the convex hull of its neighbours.
It easily follows that each Steiner point is in the convex hull $K$ of the terminals.
Therefore, each edge of $T$ has length at most $\diam K$.
Since $T$ has $2n-3$ edges, and any Steiner tree on the terminals has length at least $\diam K$, it follows that $L(T)/L(S(T))\leq 2n-3$, hence $F_2(\epsi,n)\leq 2n-4$.
\end{proof}
We will often use the following reverse triangle inequality.
\begin{lemma}\label{lemma:reversetriangle}
In $\triangle abc$, \[\length{ab}+\length{bc} \leq\frac{\length{ac}}{\cos(\theta/2)},\]
where $\theta$ is the exterior angle at $b$.
\end{lemma}
\begin{proof}
Let the angular measures of the interior angles of $\triangle abc$ at $a$, $b$, $c$, be $\alpha$, $\beta$, $\gamma$, respectively.
By the sine rule,
\begin{align*}
\frac{\length{ab}+\length{bc}}{\length{ac}} &= \frac{\sin\gamma}{\sin\beta}+\frac{\sin\alpha}{\sin\beta}
= \frac{\sin\alpha+\sin\gamma}{\sin\theta} =  \frac{2\sin\left(\frac{\alpha+\gamma}{2}\right)\cos\left(\frac{\alpha+\gamma}{2}\right)}{\sin\theta}\\
&\leq \frac{2\sin\left(\frac{\alpha+\gamma}{2}\right)}{\sin\theta} = \frac{2\sin(\theta/2)}{2\sin(\theta/2)\cos(\theta/2)} = \frac{1}{\cos(\theta/2)}.
\qedhere
\end{align*}
\end{proof}
We define a \define{cherry} of a Steiner topology $\CT$ to be a subgraph of $\CT$ consisting of two terminals with a common Steiner point.
It is easy to see that any Steiner topology on at least $3$ terminals has at least two cherries.
We will later use the fact that for any terminal $t$ there exists a cherry with two terminals not equal to $t$.
(To see this, note that in the subtree of $\CT$ on the Steiner points, there are at least two leaves, unless $n=3$.)
\begin{lemma}\label{cherry2}
Let $T$ be an $\epsi$-approximate Steiner tree in $\bR^d$, where $0\leq\epsi<2\pi/3$.
\begin{enumerate}[\textup{(}i\textup{)}]
\item For any cherry with terminals $t_1$ and $t_2$ and Steiner point $s$,
\[\length{st_1}+\length{st_2}\leq\length{t_1t_2}/\sin(\pi/3-\epsi/2).\]
\item If $D$ is the diameter of the set of terminals, then for any terminal $t$ and Steiner point $s$, \[\length{ts}\leq D\cos(\epsi/2)/\sin(\pi/3-\epsi/2).\]
\end{enumerate}
\end{lemma}
\begin{proof}
For the first statement, we use Lemma~\ref{lemma:reversetriangle}:
\[ \frac{\length{st_1}+\length{st_2}}{\length{t_1t_2}} \leq \frac{1}{\cos \frac12(\pi-\myangle t_1st_2)} = \frac{1}{\sin \frac12\myangle t_1st_2} \leq \frac{1}{\sin (\pi/3-\epsi/2)}.\]

For the second statement, consider the plane $\Uppi$ through $t$ and the terminals $t_1$, $t_2$ of a cherry (if these points are collinear, choose any plane through them).
Let $o$ be the midpoint of $t_1t_2$.
Let $C_i$ be the circle with centre $t_i$ and radius $D$.
Denote the half plane bounded by $t_1t_2$ and containing $t$ by $H$.
Let $p$ be the point where $C_1$ and $C_2$ intersect in $H$.
Without loss of generality, $t$ is inside the angle $\myangle pot_2$.

First, suppose that $\epsi\leq\pi/6$ (Fig.~\ref{fig3}(a)).
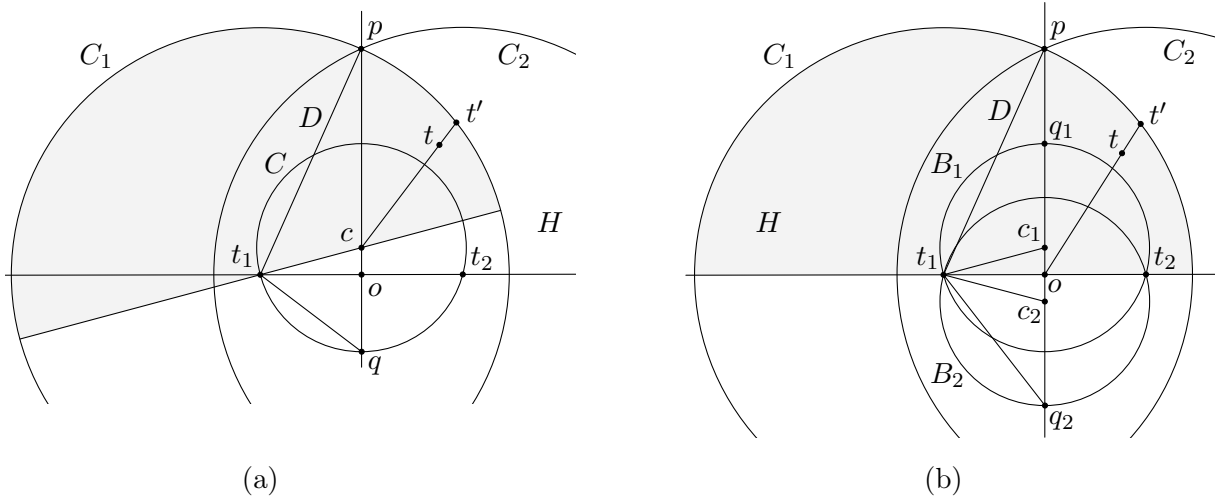
\begin{figure}
\centering
\begin{tikzpicture}[line cap=round,line join=round,>=triangle 45,x=1.0cm,y=1.0cm,scale=0.47]
\begin{scope}
\clip(-4.258373987600028,-3) rectangle (11.7,8.301842761018857);
\draw [domain=-4.258373987600028:13.51971699853897] plot(\x,{(--3.5015601753522585--0.009763683208445784*\x)/5.668636443404443});
\draw (5.723053276899828,-1.9846343289577646) -- (5.723053276899828,8.055023651793285);
\draw(5.720664273871347,1.3870164917822754) circle (2.9343064951895435cm);
\draw(2.8876541376845526,0.6226813715801264) circle (6.969866874304224cm);
\draw (8.556290581088996,0.6324450547885722) circle (6.969866874304224cm);
\draw (5.711004885180863,6.995101226088723)-- (2.8876541376845526,0.6226813715801264);
\draw (2.8876541376845526,0.6226813715801264)-- (5.725718328200734,-1.5472856508482988);
\draw (5.720664273871347,1.3870164917822754)-- (8.376507611880431,4.918205994292837);
\fill [shift={(2.8876541376845526,0.6226813715801264)},fill opacity=0.05]  (0,0) --  plot[domain=0.2635217902941258:3.4051144438839187,variable=\t]({1.*6.969866874304225*cos(\t r)+0.*6.969866874304225*sin(\t r)},{0.*6.969866874304225*cos(\t r)+1.*6.969866874304225*sin(\t r)}) -- cycle ;
\draw (-3.84,-1.19) -- (9.62,2.44); 
\draw [fill=black] (2.8876541376845526,0.6226813715801264) circle (2pt);
\draw[color=black] (2.4,1.1513487257037989) node {$t_1$};
\draw [fill=black] (8.556290581088996,0.6324450547885722) circle (2pt);
\draw[color=black] (9.1,1.0836656381931178) node {$t_2$};
\draw[color=black] (11,2.1) node {$H$};
\draw [fill=black] (5.721972359386775,0.6275632131843493) circle (2pt);
\draw [color=black] (6.1,0.2) node {$o$};
\draw [fill=black] (5.720664273871347,1.3870164917822754) circle (2pt);
\draw[color=black] (5.3,1.7379354841297028) node {$c$};
\draw[color=black] (3.3221318135962963,3.7684281094501397) node {$C$};
\draw[color=black] (-1.7,6.8) node {$C_1$};
\draw[color=black] (10,6.8) node {$C_2$};
\draw [fill=black] (5.711004885180863,6.995101226088723) circle (2pt);
\draw[color=black] (6.1,7.5) node {$p$};
\draw[color=black] (4.3,5.07696780132331) node {$D$};
\draw [fill=black] (5.725718328200734,-1.5472856508482988) circle (2pt);
\draw[color=black] (6.1,-1.95) node {$q$};
\draw [fill=black] (7.902020735152411,4.287331780365361) circle (2pt);
\draw[color=black] (7.6,4.558064130408087) node {$t$};
\draw [fill=black] (8.376507611880431,4.918205994292837) circle (2pt);
\draw[color=black] (8.9,5.2348950055148995) node {$t'$};
\end{scope}
\draw [fill=black] (2.8876541376845526,-5.2) node {(a)};
\end{tikzpicture}
\hfill
\begin{tikzpicture}[line cap=round,line join=round,>=triangle 45,x=1.0cm,y=1.0cm,scale=0.47]
\begin{scope}
\clip(-4.325492366311542,-3.960032687409054) rectangle (10.486647093449143,8.301842761018857);
\draw [domain=-4.325492366311542:10.486647093449143] plot(\x,{(--3.5015601753522585--0.009763683208445784*\x)/5.668636443404443});
\draw (5.723053276899828,-3.960032687409054) -- (5.723053276899828,8.301842761018857);
\draw(5.720664273871347,1.3870164917822754) circle (2.9343064951895435cm);
\draw(2.8876541376845526,0.6226813715801264) circle (6.969866874304224cm);
\draw(8.556290581088996,0.6324450547885722) circle (6.969866874304224cm);
\draw (5.711004885180863,6.995101226088723)-- (2.8876541376845526,0.6226813715801264);
\draw(5.723280444902204,-0.13189006541357673) circle (2.9343064951895435cm);
\draw (5.721972359386775,0.6275632131843493)-- (8.407351606882106,4.878499047897742);
\fill [shift={(2.8876541376845526,0.6226813715801264)},fill opacity=0.05]  (0,0) --  plot[domain=0.001722402494976416:3.1433150560847696,variable=\t]({1.*6.969866874304224*cos(\t r)+0.*6.969866874304224*sin(\t r)},{0.*6.969866874304224*cos(\t r)+1.*6.969866874304224*sin(\t r)}) -- cycle ;
\draw (5.720664273871347,1.3870164917822754)-- (2.8876541376845526,0.6226813715801264);
\draw (2.8876541376845526,0.6226813715801264)-- (5.723280444902204,-0.1318900654135768);
\draw (2.8876541376845526,0.6226813715801264)-- (5.72833449923159,-3.0661922080441513);
\draw [fill=black] (2.8876541376845526,0.6226813715801264) circle (2pt);
\draw[color=black] (2.4494514213703016,1.1) node {$t_1$};
\draw [fill=black] (8.556290581088996,0.6324450547885722) circle (2pt);
\draw[color=black] (9.1,1.1) node {$t_2$};
\draw[color=black] (-2,2.1) node {$H$};
\draw [fill=black] (5.721972359386775,0.6275632131843493) circle (2pt);
\draw[color=black] (6,0.3) node {$o$};
\draw [fill=black] (5.720664273871347,1.3870164917822754) circle (2pt);
\draw[color=black] (5.3,1.8) node {$c_1$};
\draw[color=black] (3.0,3.7) node {$B_1$};
\draw[color=black] (-1.7,6.8) node {$C_1$};
\draw[color=black] (9.5,6.9) node {$C_2$};
\draw [fill=black] (5.711004885180863,6.995101226088723) circle (2pt);
\draw[color=black] (6.1,7.5) node {$p$};
\draw[color=black] (4.458750339390011,5.262134141450594) node {$D$};
\draw [fill=black] (5.723280444902204,-0.1318900654135768) circle (2pt);
\draw[color=black] (5.3,-0.5) node {$c_2$};
\draw[color=black] (3.0,-2.2) node {$B_2$};
\draw [fill=black] (7.884862597038998,4.05140274213103) circle (2pt);
\draw[color=black] (7.653020414190577,4.437806380211743) node {$t$};
\draw [fill=black] (5.71561021954196,4.321318634412849) circle (2pt);
\draw[color=black] (6.2,4.7) node {$q_1$};
\draw [fill=black] (5.72833449923159,-3.0661922080441513) circle (2pt);
\draw[color=black] (6.2,-3.5) node {$q_2$};
\draw [fill=black] (8.407351606882106,4.878499047897742) circle (2pt);
\draw[color=black] (8.9,5.1848534138344515) node {$t'$};
\end{scope}
\draw [fill=black] (2.8876541376845526,-5.2) node {(b)};
\end{tikzpicture}
\caption{Proof of Lemma~\ref{cherry2}}\label{fig3}
\end{figure}
Let $c$ be the point on the line $op$ in the half plane $H$ such that $\myangle ct_1t_2=\pi/6-\epsi$.
Let $t'$ be the point where the ray from $c$ through $t$ intersects $C_1$.
Then $\length{ct}\leq\length{ct'}\leq\length{cp}$ (Euclid III.7 \cite{Euclid}).
Let $C$ be the circle with centre $c$ that passes through $t_1$ and $t_2$, and let it intersect the line $op$ in the half plane opposite $H$ in $q$.
Then for any point $x\in C\cap H$, $\myangle t_1xt_2=\pi/3+\epsi$ and for any $x\in C\setminus H$, $\myangle t_1xt_2=2\pi/3-\epsi$.
Since $\pi/3+\epsi<2\pi/3-\epsi\leq \myangle t_1st_2$, $s$ is in the ball $B$ with centre $c$ that passes through $t_1$ and $t_2$.
In particular, $\length{cs}\leq\length{cq}$.
We conclude that
\begin{equation}\label{eq:dist}
\length{ts}\leq\length{tc}+\length{cs}\leq\length{pc}+\length{cq}=\length{pq}=D\frac{\sin\myangle pt_1q}{\sin\myangle t_1qp}.
\end{equation}
We bound $\myangle pt_1q$ from below as follows.
Since $\length{t_1t_2}\leq D=\length{t_1p}$, $\myangle pt_1t_2\geq\pi/3$.
Also, $\myangle qt_1t_2=\pi/6+\epsi/2$.
Therefore, $\myangle pt_1q\geq(\pi+\epsi)/2$.
We substitute this estimate, together with $\myangle t_1qp=\pi/3-\epsi/2$ into \eqref{eq:dist}, to obtain \[\length{ts} \leq D\frac{\sin(\pi/2+\epsi/2)}{\sin(\pi/3-\epsi/2)} = D\frac{\cos(\epsi/2)}{\sin(\pi/3-\epsi/2)}.\]

The case where $\epsi>\pi/6$ is similar (Fig.~\ref{fig3}(b)).
Let $c_1$ and $c_2$ be points on the line $op$ such that $\myangle c_it_1o=\epsi-\pi/6$, $i=1,2$.
Similar to the previous case, $\length{ot}\leq\length{op}$.

Let $B_i$ be the ball with centre $c_i$ and radius $\length{c_1t_1}=\length{c_2t_1}$, $i=1,2$.
Let $q_1$ be the point where the line $oc_1$ intersects $B_1$ in the half plane $H$, and $q_2$ the point where $oc_1$ intersects $B_2$ in the half plane opposite $H$. 
Since $B_1\cup B_2$ is the set of all points $x$ such that $\myangle t_1xt_2\geq2\pi/3-\epsi$, $s\in B_1\cup B_2$.
If $s\in B_1$, then Euclid III.7 gives that $\length{os}\leq\length{oq_1}=\length{oq_2}$.
It follows that $\length{st}=\length{so}+\length{ot}\leq\length{q_2o}+\length{op}=\length{pq_2}$.
Similar to the previous case, we obtain
\[ \length{pq_2} = D\frac{\sin\myangle pt_1q_2}{\sin\myangle t_1q_2p} \leq D\frac{\sin(\pi/2+\epsi/2)}{\sin(\pi/3-\epsi/2)}= D\frac{\cos(\epsi/2)}{\sin(\pi/3-\epsi/2)}. \qedhere\]
\end{proof}
\begin{theorem}\label{upperbound}
For any $\epsi\in(0,2\pi/3)$ and $d\geq 2$, \[F_d(\epsi,n)=O\left(\left(\cos(\epsi/2)/\sin(\pi/3-\epsi/2)\right)^n\right).\]
\end{theorem}
\begin{proof}
Let $A=\cos(\epsi/2)/\sin(\pi/3-\epsi/2)$ and $B=1/\sin(\pi/3-\epsi/2)$.
We show by induction on $n\geq 2$ that
\begin{equation}\label{ind}
L(T)\leq \left(A^{n-2}+\frac{(A^{n-2}-1)B}{A-1}\right)D.
\end{equation}
If $n=2$, $L(T)=D$, which equals the right-hand side.
Next let $n>2$ and assume that~\eqref{ind} holds for $\epsi$-approximate Steiner trees on $n-1$ terminals.
Consider a cherry of $T$ with Steiner point $s$ and terminals $t_1$ and $t_2$.
By Lemma~\ref{cherry2}, the distance between $s$ and any terminal of $T$ is at most $AD$, and $\length{st_1}+\length{st_2}\leq B\length{t_1t_2}\leq BD$.
Remove $t_1$ and $t_2$ and the edges $st_1$ and $st_2$ from $T$ and change $s$ into a terminal to obtain an $\epsi$-approximate Steiner tree $T'$ on $n-1$ terminals.
The diameter of this set of terminals is $D'\leq AD$.
By the induction hypothesis, $L(T')\leq \left(A^{n-3}+\frac{(A^{n-3}-1)B}{A-1}\right)D'$.
Therefore,
\begin{align*}
L(T) &= L(T') +\length{st_1}+\length{st_2} \\
     &\leq \left(A^{n-3}+\frac{(A^{n-3}-1)B}{A-1}\right)AD + BD \\
     &= \left(A^{n-2}+\frac{(A^{n-2}-1)B}{A-1}\right)D.
\end{align*}
Finally, the length of a Steiner minimal tree joining the terminals of $T$ is at least $D$, and it follows that
\[ \frac{L(T)}{L(S(T))}-1\leq A^{n-2}+\frac{(A^{n-2}-1)B}{A-1}-1 = O(A^n).\qedhere\]
\end{proof}

The following is a sharper version of Lemma~2.2 in \cite{RWW}.
The proof is along the lines of the proof of Lemma~2.2 in \cite{RWW}, but is done in the plane.
\begin{theorem}\label{polynomial}
For each $\epsi\in(\pi/3,2\pi/3)$, $F_2(\epsi,n) =\upOmega\left(n^{\log_2 C_2(\epsi)}\right)$ 
where $C_2(\epsi)=\bigl(2\sin(\frac{\pi}{3}-\frac{\epsi}{2})\bigr)^{-1}$.
Also, $F_2(\pi/3, n)=\upOmega(\log n)$.
\end{theorem}
By making $\epsi$ large enough, the lower bound in Theorem~\ref{polynomial} grows faster than any polynomial.
In particular, if $\epsi>105.6\dots^\circ$, then the lower bound is superquadratic (compare with Theorem~2.1 in \cite{RWW}).
Theorem~\ref{polynomial} follows from the following lemma (combined with Proposition~\ref{monotone}).
\begin{lemma}
Let $k\geq 1$ and $\pi/3<\epsi<2\pi/3$.
Then $F_2(\epsi,2^{k+1}) > \sqrt{3}\frac{C^k-1}{C-1}-1$, 
where $C=\bigl(2\sin(\frac{\pi}{3}-\frac{\epsi}{2})\bigr)^{-1}$.
Also, for $k\geq 1$, $F_2(\pi/3, 2^{k+1})\geq \sqrt{3}k-1$.
\end{lemma}
\begin{proof}
Let $\pi/3\leq \epsi<2\pi/3$ and $k\geq 1$.
We construct an $\epsi$-approximate Steiner tree with $2^{k+1}$ terminals.
Let $r=\sin(\frac{\pi}{3}-\frac{\epsi}{2})$.
Let $C_0,C_1,\dots,C_k$ be concentric circles with common centre $o$ and with $C_i$ of radius $r^i$.

First, we construct ``half'' the tree with $2^k$ terminals on $C_k$ and Steiner points on the other circles.
Fix any $p_1\in C_0$.
There are two tangent lines from $p_1$ to $C_1$.
Denote the points where they touch $C_1$ by $p_2$ and $p_3$, chosen such that $\myangle p_2p_1p_3$ is positively oriented.
See Figure~\ref{circlefig}.
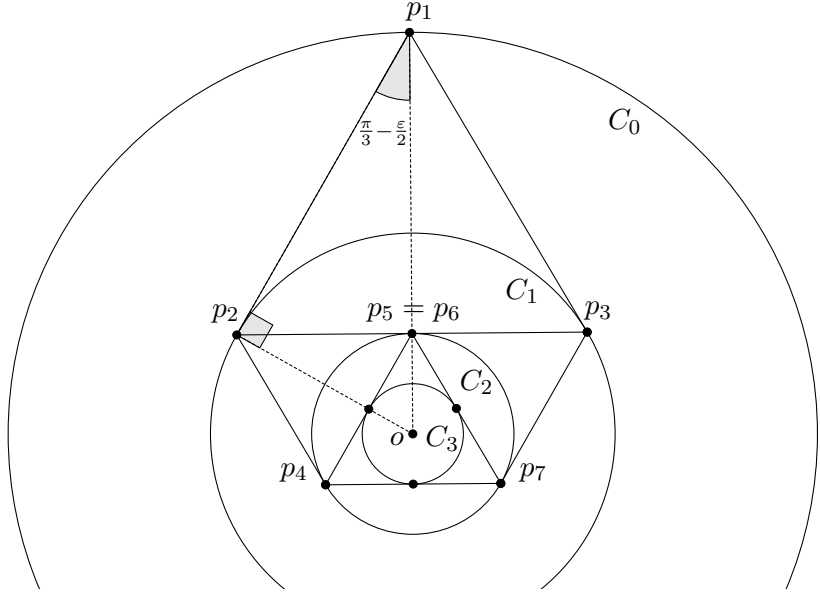
\begin{figure}
\centering
\begin{tikzpicture}[line cap=round,line join=round,>=triangle 45,x=1.0cm,y=1.0cm, scale=1.5]
\clip(-0.9950503580131429,-0.025751107804548312) rectangle (6.548804899716984,5.196917916777834);
\draw [shift={(2.711025283914615,4.889402720459528)},fill=black,fill opacity=0.1] (0,0) -- (-119.53228973996148:0.6) arc (-119.53228973996148:-89.53228973996147:0.6) -- cycle;
\draw[fill=black,fill opacity=0.1] (1.3979069594244686,2.100318487445869) -- (1.512392731994789,2.3024055879814576) -- (1.3103056314592003,2.416891360551778) -- (1.19581985888888,2.2148042600161895) -- cycle; 
\draw(2.74,1.34) circle (3.5495209826366327cm);
\draw(2.74,1.34) circle (1.7747604913183164cm);
\draw(2.74,1.34) circle (0.8873802456591582cm);
\draw(2.74,1.34) circle (0.4436901228295791cm);
\draw [dash pattern=on 1pt off 1pt] (2.711025283914615,4.889402720459528)-- (2.74,1.34);
\draw [dash pattern=on 1pt off 1pt] (2.74,1.34)-- (1.19581985888888,2.2148042600161895);
\draw [dash pattern=on 1pt off 1pt] (1.19581985888888,2.2148042600161895)-- (2.711025283914615,4.889402720459528);
\draw (2.711025283914615,4.889402720459528)-- (1.19581985888888,2.2148042600161895);
\draw (2.711025283914615,4.889402720459528)-- (4.269692783068428,2.2398971002135744);
\draw (1.19581985888888,2.2148042600161895)-- (1.9751536084657872,0.8900514498932124);
\draw (1.19581985888888,2.2148042600161895)-- (2.7327563209786545,2.2273506801148817);
\draw (4.269692783068428,2.2398971002135744)-- (2.7327563209786545,2.2273506801148817);
\draw (4.269692783068428,2.2398971002135744)-- (3.5120900705555593,0.9025978699919045);
\draw (1.9751536084657872,0.8900514498932124)-- (2.3539549647222158,1.5587010650040496);
\draw (1.9751536084657872,0.8900514498932124)-- (2.743621839510673,0.8963246599425588);
\draw (3.5120900705555593,0.9025978699919045)-- (3.1224231957671087,1.5649742750533941);
\draw (3.5120900705555593,0.9025978699919045)-- (2.743621839510673,0.8963246599425588);
\draw (2.7327563209786545,2.2273506801148817)-- (2.3539549647222158,1.5587010650040496);
\draw (2.7327563209786545,2.2273506801148817)-- (3.1224231957671087,1.5649742750533941);
\draw [fill=black] (2.74,1.34) circle (1pt);
\draw[color=black] (2.6,1.3) node {$o$};
\draw[color=black] (4.6,4.1) node {$C_0$};
\draw[color=black] (3.7,2.6) node {$C_1$};
\draw[color=black] (3.3,1.8) node {$C_2$};
\draw[color=black] (3,1.3) node {$C_3$};
\draw [fill=black] (2.711025283914615,4.889402720459528) circle (1pt);
\draw[color=black] (2.8042497500163988,5.076479008454132) node {$p_1$};
\draw [fill=black] (1.19581985888888,2.2148042600161895) circle (1pt);
\draw[color=black] (1.1,2.4049250420010893) node {$p_2$};
\draw [fill=black] (4.269692783068428,2.2398971002135744) circle (1pt);
\draw [fill=black] (1.9751536084657872,0.8900514498932124) circle (1pt);
\draw [fill=black] (3.5120900705555593,0.9025978699919045) circle (1pt);
\draw [fill=black] (2.7327563209786545,2.2273506801148817) circle (1pt);
\draw[color=black] (2.75,2.4158740336668796) node {$p_5=p_6$};
\draw [fill=black] (4.269692783068428,2.239897100213574) circle (1pt);
\draw[color=black] (4.369955558224539,2.426823025332671) node {$p_3$};
\draw [fill=black] (1.9751536084657872,0.8900514498932124) circle (1pt);
\draw[color=black] (1.7,1) node {$p_4$};
\draw [fill=black] (3.5120900705555593,0.9025978699919045) circle (1pt);
\draw[color=black] (3.8,1) node {$p_7$};
\draw[color=black] (2.48,4) node {$\scriptstyle \frac{\pi}{3}-\frac{\epsi}{2}$};
\draw [fill=black] (2.3539549647222158,1.5587010650040496) circle (1pt);
\draw [fill=black] (3.1224231957671087,1.5649742750533941) circle (1pt);
\draw [fill=black] (2.743621839510673,0.8963246599425588) circle (1pt);
\end{tikzpicture}
\caption{Constructing a lower bound in the plane}\label{circlefig}
\end{figure}
Note that $\myangle p_2p_1p_3=2\pi/3-\epsi$.

In general, for each $i=1,\dots,k$, once $p_{2^{i-1}},p_{2^{i-1}+1},\dots,p_{2^i-1}\in C_{i-1}$ have been determined, for each $p_j\in C_{i-1}$, let $p_{2j}$ and $p_{2j+1}$ be the two points where the tangents from $p_j$ touch $C_i$, chosen such that $\myangle p_{2j}p_jp_{2j+1}$ is positively oriented.
Again, $\myangle p_{2j}p_jp_{2j+1}=2\pi/3-\epsi$.
The points $p_{2^k},\dots,p_{2^{k+1}-1}\in C_k$ will be $2^k$ of the terminals.
We join each $p_j$ to $p_{2j}$ and $p_{2j+1}$, for $j=1,\dots,2^k-1$.

Next, we ``double'' the tree, by choosing one of the directions on the tangent line of $C_0$ at $p_1$, and moving each $p_i$ in that direction by a distance of $\delta$, where $\delta>0$ is very small.
Denote the moved points by $p'_i$.
We move $o$ in the same direction to obtain $o'$.
The moved points $p'_{2^k},\dots,p'_{2^{k+1}-1}$ will give another $2^k$ terminals.
We join $p'_j$ to $p'_{2j}$ and $p'_{2j+1}$, for $j=1,\dots,2^k-1$.
Finally, we join $p_1$ and $p_1'$.
All $p_j$ and $p'_j$ with $j<2^k$ are Steiner points.
Each angle at a Steiner point is one of three values $2\pi/3-\epsi$, $5\pi/6-\epsi/2$, and $\pi/6+\epsi/2$.
These all belong to the interval $[2\pi/3-\epsi,2\pi/3+\epsi]$, since $\epsi\geq\pi/3$.
Thus, we obtain a full $\epsi$-approximate Steiner tree $T$ on $2^{k+1}$ terminals, all on the circle $C_k$ of radius~$r^k$.
Note that many of the $p_j$ coincide.
For instance, it is always the case that~$p_5=p_6$.
This is allowed in our definition of an $\epsi$-approximate Steiner tree.
Alternatively, we could have slightly perturbed the radii of the circles by $\delta$ to ensure that all $p_j$ are distinct.

Next, we calculate $L(T)$.
An edge from a point of $T$ on $C_{i}$ to a point on $C_{i+1}$ has length $r^i\cos(\frac{\pi}{3}-\frac{\epsi}{2})$.
Therefore,
\begin{align*}
L(T) &= \delta + 2\left(2\cos(\tfrac{\pi}{3}-\tfrac{\epsi}{2})+4r\cos(\tfrac{\pi}{3}-\tfrac{\epsi}{2})+\dots +2^kr^{k-1}\cos(\tfrac{\pi}{3}-\tfrac{\epsi}{2})\right) \\
&= \delta + 4\cos(\tfrac{\pi}{3}-\tfrac{\epsi}{2})\left(1+2r+\dots+(2r)^{k-1}\right) \\
&= \delta + 4\cos(\tfrac{\pi}{3}-\tfrac{\epsi}{2})\frac{1-(2r)^k}{1-2r}
\end{align*}
if $\epsi>\pi/3$, and $L(T)=\delta+2\sqrt{3}k$ if $\epsi=\pi/3$.
We form a Steiner tree $S$ with a degeneration of the topology of $T$ by joining each $p_i$ to $o$, each $p_i'$ to $o'$, and $o$ to $o'$.
Then $L(S(T))\leq L(S)=\delta+2 (2r)^{k}$, which equals $\delta+2$ if $\epsi=\pi/3$.

Therefore, if $\epsi>\pi/3$,
\[ F_2(\epsi,2^{k+1}) \geq \frac{\delta + 4\cos(\frac{\pi}{3}-\frac{\epsi}{2})\dfrac{1-(2r)^k}{1-2r}}{\delta+2(2r)^{k}}-1\]
for each $\delta>0$, hence
\begin{align*}
F_2(\epsi,2^{k+1}) &\geq \frac{4\cos(\frac{\pi}{3}-\frac{\epsi}{2})\left(1-(2r)^k)\right)}{2(2r)^{k}(1-2r)} - 1 \\
&= \frac{2\cos(\frac{\pi}{3}-\frac{\epsi}{2})}{1-2r}\left(\left(\frac{1}{2r}\right)^k-1\right)-1 \\
&= \frac{2\sqrt{1-\frac{1}{4C^2}}}{1-\frac{1}{C}}(C^k-1)-1 \\
&= \frac{\sqrt{4C^2-1}}{C-1}(C^k-1)-1 > \sqrt{3}\frac{C^k-1}{C-1}-1,
\end{align*}
where $C=1/(2r)=(2\sin(\frac{\pi}{3}-\frac{\epsi}{2}))^{-1}$.
Similarly, if $\epsi=\pi/3$, then \[F_2(\epsi,2^{k+1})\geq \frac{\delta+2\sqrt{3}k}{2+\delta}-1,\] and letting $\delta\to 0$, we obtain the required result.
\end{proof}

\section{\texorpdfstring{Upper Bounds for Small $\epsi$ (Proof of Theorem~\ref{thm:plane})}{Upper Bounds for Small Epsilon (Proof of Theorem~\ref{thm:plane})}}\label{section:small}
In this section, we prove Theorem~\ref{thm:plane} using an unfolding algorithm described in \cite{BTW} and \cite{RWW} based on Melzak's algorithm for finding the shortest Steiner tree for a fixed Steiner topology (if this shortest tree happens to be what we call a locally minimum Steiner tree).
This algorithm unfolds an approximate Steiner tree into a broken line segment.
First, we describe this unfolding and then use it in the special cases of $3$ and $4$ terminals in the plane to determine the exact values of $F_2(\epsi,3)$ and $F_2(\epsi,4)$ (Propositions~\ref{prop:3unfold} and \ref{prop:4unfold}).
Then the proof of Theorem~\ref{thm:plane} should be clear.

The following inequality and its proof forms the basis for the unfolding algorithm.
\begin{lemma}\label{lemma:triangle}
Let $\triangle abc$ be an equilateral triangle in $\bR^d$.
Then for any $x\in\bR^d$, $\length{xa}\leq\length{xb}+\length{xc}$, with equality iff $x$ is on the minor arc $\arc{bc}$ of the circumcircle of $\triangle abc$.
\end{lemma}
\begin{proof}
The proof is essentially the same as the classical proof that the Fermat point of a triangle with all angles less than $2\pi/3$ minimizes the sum of the distances to the vertices.
Because there are only $4$ points to consider, we may assume without loss of generality that~$x,a,b,c\in\bR^3$.

Rotate $\triangle bxc$ by an angle of $\pi/3$ around the axis through $b$ perpendicular to the plane $\Uppi$ through $a$, $b$ and $c$, such that $c$ is rotated to $a$.
Then $b$ stays fixed, and $x$ is rotated to $x'$, say.
Also, $\length{xc}=\length{x'a}$.
Let $p\colon\bR^3\to\Uppi$ be the orthogonal projection onto $\Uppi$.
Then $\triangle bp(x)p(x')$ is equilateral.
Since $xx'$ is parallel to $\Uppi$, $\length{xx'}=\length{p(x)p(x')}=\length{bp(x)}\leq \length{bx}$.
Therefore, $\length{xa}\leq\length{xx'}+\length{x'a}\leq\length{bx}+\length{xc}$.
Equality holds iff $x$ is in the plane $\Uppi$ and $a$, $x'$, $x$ are collinear, which holds iff $\myangle bx'a=2\pi/3$, iff $\myangle bxc=2\pi/3$, iff $x$ is on the minor arc $\arc{bc}$ of the circumcircle of $\triangle abc$.
\end{proof}

Consider a family of $n$ terminals $N_n$ in $\bR^d$ and a full Steiner topology $\CT_n$ for those terminals.
Choose one of the terminals $t_0$ as root of $\CT_n$.
We define a \define{Melzak sequence} of $N_n$ and $\CT_n$ to be two sequences $N_n,N_{n-1},\dots,N_2$ and $\CT_n,\CT_{n-1},\dots,\CT_2$, where each $\CT_i$ is a full Steiner topology on $N_i$ and with root $t_0$ (thus $t_0\in N_i$ for all $i$), and where we obtain $N_{i-1}$ and $\CT_{i-1}$ from $N_i$ and $\CT_i$ as follows: choose any cherry of $\CT_i$ with two terminals $t_1, t_2\neq t_0$ and Steiner point $s$ with neighbours $t_1$, $t_2$ and $p$, say, replace $t_1$ and $t_2$ in $N_{i}$ by any point $t\in\bR^d$ such that $\triangle t_1t_2t$ is an equilateral triangle, thus obtaining $N_{i-1}$, and remove $s$ and its incident edges from $\CT_{i}$ and replace them by the edge $pt$, to obtain $\CT_{i-1}$.
If $N_2=\set{t_0,t}$, say, then we call the line segment $t_0t$ an \define{unfolding} of $N_n$ with respect to the topology $\CT_n$.

It is not difficult to see that if there is more than one cherry to choose from at a certain stage, it does not matter which we choose first: we may in fact process both cherries in parallel.
(This is equivalent to saying that in the subtree of $\CT_n$ on the Steiner points, it does not matter in which order we remove leaves, and that this may be done in parallel.)

Lemma~\ref{lemma:triangle} and induction immediately gives the following, which is Theorem~3.1 of \cite{RWW} and Theorem~4.2 of \cite{BTW}:
\begin{lemma}\label{lemma:unfolding}
The length of any unfolding of a terminal set $N_n\subset\bR^d$ with respect to a full Steiner topology $\CT_n$ is a lower bound for the shortest tree on $N_n$ which has $\CT_n$ as topology \textup{(}allowing degenerate topologies\textup{)}.
\end{lemma}

Next, we describe the plan of the proof of Theorem~\ref{thm:plane}.
First, we unfold a planar $\epsi$-approximate Steiner tree into a polygonal path of the same length, and estimate the \emph{turn} at each internal vertex of the path.
By Lemma~\ref{lemma:unfolding}, the length between the endpoints of the unfolding is a lower bound on the length of a Steiner minimal tree on the same terminal set.
By a result of E.~Schmidt \cite{Schmidt} (Lemma~\ref{schmidt} below), this length is minimised among all polygonal paths with the same angles and edges of the same length, by a planar, convex path.
Finally, we minimise the length of the endpoint among all polygonal paths of the same total length and the same sum of turns.

Before providing the detail of the general case, we show how to determine exact values for small $n$.

\begin{proposition}\label{prop:3unfold}
For all $\epsi\in(0,\pi/3)$, $F_2(\epsi,3)=G_2(\epsi,3)=\frac{1}{\cos \epsi/2}-1$.
\end{proposition}

\begin{proof}
We show that $F_2(\epsi,3)\leq(\cos \epsi/2)^{-1}-1$.
Consider an $\epsi$-approximate Steiner tree $T$ on three terminals $t_0$, $t_1$, $t_2$ in the plane, with Steiner point $s$ and edges $e_i=st_i$, $i=0,1,2$, numbered in such a way that $e_0$, $e_1$, $e_2$ are in anti-clockwise order around $s$.
See Fig.~\ref{fig:3unfold}.
\begin{figure}
\centering
\definecolor{ffqqqq}{rgb}{1,0,0}
\begin{tikzpicture}[line cap=round,line join=round,>=triangle 45,x=1.0cm,y=1.0cm, scale=0.6]
\fill[color=ffqqqq,fill=ffqqqq,fill opacity=0.05] (1.27,-0.38) -- (7.66,0.52) -- (5.25,-5.46) -- cycle;
\fill[color=ffqqqq,fill=ffqqqq,fill opacity=0.05] (1.27,-0.38) -- (4.66,2.8) -- (5.72,-1.73) -- cycle;
\draw [shift={(4.66,2.8)},fill=black,fill opacity=0.05] (0,0) -- (-136.81:0.67) arc (-136.81:-37.19:0.67) -- cycle;
\draw (2.92,5.05)-- (4.66,2.8);
\draw (4.66,2.8)-- (7.66,0.52);
\draw (1.27,-0.38)-- (4.66,2.8);
\draw [line width=1.2pt,dash pattern=on 3pt off 3pt] (5.72,-1.73)-- (4.66,2.8);
\draw [line width=1.2pt,dash pattern=on 3pt off 3pt] (5.72,-1.73)-- (5.25,-5.46);
\draw [color=ffqqqq] (1.27,-0.38)-- (7.66,0.52);
\draw [color=ffqqqq] (7.66,0.52)-- (5.25,-5.46);
\draw [color=ffqqqq] (5.25,-5.46)-- (1.27,-0.38);
\draw [color=ffqqqq] (1.27,-0.38)-- (4.66,2.8);
\draw [color=ffqqqq] (4.66,2.8)-- (5.72,-1.73);
\draw [color=ffqqqq] (5.72,-1.73)-- (1.27,-0.38);
\fill [color=black] (2.92,5.05) circle (3pt);
\draw[color=black] (2.5,5.1) node {$t_0$};
\fill [color=black] (1.27,-0.38) circle (3pt);
\draw[color=black] (0.8,-0.1) node {$t_1$};
\fill [color=black] (7.66,0.52) circle (3pt);
\draw[color=black] (8.2,0.6) node {$t_2$};
\fill [color=black] (4.7,2.8) circle (3pt);
\draw[color=black] (4.9,3.11) node {$s$};
\draw (3.4,2.9) node {$\scriptstyle\frac{2\pi}{3}+\epsi_1$};
\draw (6.2,2.9) node {$\scriptstyle\frac{2\pi}{3}+\epsi_2$};
\draw (5.0,1.7) node {$\scriptstyle\frac{2\pi}{3}+\epsi_3$};
\draw[color=black] (3.3,3.85) node {$e_0$};
\draw[color=black] (6.9,1.7) node {$e_2$};
\draw[color=black] (2.3,1.3) node {$e_1$};
\fill [color=black] (5.72,-1.73) circle (3pt);
\draw[color=black] (6.2,-1.39) node {$t'_1$};
\fill [color=black] (4.66,2.8) circle (3pt);
\draw[color=black] (4.8,0.6) node {$e'_1$};
\fill [color=black] (5.25,-5.46) circle (3pt);
\draw[color=black] (5.81,-5.32) node {$t'_2$};
\fill [color=black] (5.72,-1.73) circle (3pt);
\fill [color=black] (5.25,-5.46) circle (3pt);
\draw[color=black] (5,-3.34) node {$e'_2$};
\end{tikzpicture}
\caption{Unfolding an $\epsi$-approximate Steiner tree on $3$ terminals}\label{fig:3unfold}
\end{figure}
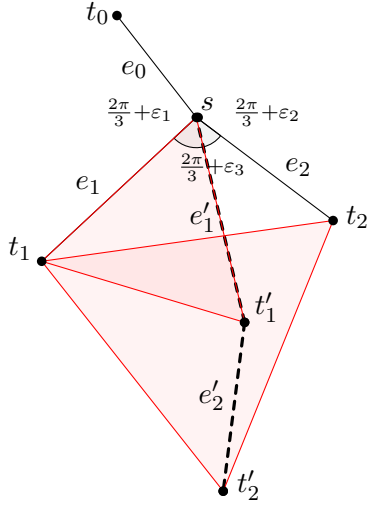
Let $\myangle t_0st_1=2\pi/3+\epsi_1$, $\myangle t_0st_2=2\pi/3+\epsi_2$ and $\myangle t_1st_2=2\pi/3+\epsi_3$, where $\abs{\epsi_i}\leq\epsi$, $i=1,2,3$.
Since $\epsi\leq\pi/3$, the three angles sum to $2\pi$, and $\epsi_1+\epsi_2+\epsi_3=0$.

We unfold the tree into a polygonal line of total length $L(T)$ as follows.
We rotate $e_1=st_1$ by an angle of $\pi/3$ around $s$ to obtain the edge $e_1'=st_1'$, say.
We rotate $e_2=st_2$ by an angle of $-\pi/3$ around $t_1$ to obtain the edge $e_2'=t_1't_2'$.
Then $t_0st_1't_2'$ is a polygonal line of length $L(T)$ (see Fig.~\ref{fig:3unfold}).
The turn from edge $e_0$ to $e_1'$ equals $\epsi_1$, and the turn from $e_1'$ to $e_2'$ equals $\epsi_3$.
Since $\abs{\epsi_1+\epsi_3}=\abs{\epsi_2}\leq\epsi<\pi$, the rays $\Ray{t_0s}$ and $\Ray{t_2't_1'}$ intersect in $p$, say.
Then $L(T)=\length{t_0s}+\length{st_1'}+\length{t_1't_2'}\leq\length{t_0p}+\length{pt_2'}$.
By Lemma~\ref{lemma:unfolding}, $L(S(T))\geq\length{t_0t_2'}$.
It follows that
\[ \frac{L(T)}{L(S(T))}\leq\frac{\length{t_0p}+\length{pt_2'}}{\length{t_0t_2'}}\leq\frac{1}{\cos{\epsi/2}}\quad\text{by Lemma~\ref{lemma:reversetriangle},}\]
and \[F_2(\epsi,3)=\sup \frac{L(T)}{L(S(T))}-1\leq\frac{1}{\cos{\epsi/2}}-1.\]

To show that $G_2(\epsi,3)\geq (\cos{\epsi/2})^{-1}-1$, consider an $\epsi$-approximate tree $T$ as above with $\epsi_1=\epsi_2=-\epsi/2$, $\epsi_3=\epsi$, $\length{t_0s}=\delta$ for arbitrarily small $\delta>0$, and $\length{t_1s}=\length{t_2s}=1$.
Then $L(T)=2+\delta$ and $L(S(T))=\delta+2\cos(\epsi/2)$.
Since all angles in $\triangle t_0t_1t_2$ are less than $2\pi/3$ if $\delta$ is small enough, $S(T)$ is not degenerate, hence $G_2(\epsi,3)\geq\frac{2+\delta}{\delta+2\cos\epsi/2}-1$ for all~$\delta>0$.
It follows that $G_2(\epsi,3)\geq(\cos \epsi/2)^{-1}-1$.
\end{proof}

\begin{proposition}\label{prop:4unfold}
For all $\epsi\in(0,\pi/3)$, $F_2(\epsi,4)=G_2(\epsi,4)=\frac{1}{\cos \epsi}-1$.
\end{proposition}

\begin{proof}
Consider an $\epsi$-approximate Steiner tree on four terminals $t_1$, $t_2$, $t_3$, $t_4$, Steiner points $s_1$ and $s_2$, and edges $e_1=s_1t_1$, $e_2=s_1t_2$, $e_0=s_1s_2$, $e_3=s_2t_3$, $e_4=s_2t_4$, labelled in such a way that $e_0$, $e_1$, $e_2$ are in anti-clockwise order around $s_1$, and $e_0$, $e_4$, $e_3$ are in anti-clockwise order around $s_2$.
Also, let $\myangle t_1s_1t_2=2\pi/3+\epsi_1$, $\myangle t_1s_1s_2=2\pi/3+\epsi_2$, $\myangle s_1s_2t_4=2\pi/3+\epsi_3$ and $\myangle t_3s_2t_4=2\pi/3+\epsi_4$, where $\abs{\epsi_i}\leq\epsi$, $i=1,2,3,4$, and $\abs{\epsi_1+\epsi_2}, \abs{\epsi_3+\epsi_4}\leq\epsi$.
See Fig.~\ref{fig:4unfold}.
\begin{figure}
\definecolor{ffqqqq}{rgb}{1,0,0}
\definecolor{qqqqff}{rgb}{0,0,1}
\definecolor{xfqqff}{rgb}{0.4980392156862745,0.,1.}
\definecolor{uuuuuu}{rgb}{0.26666666666666666,0.26666666666666666,0.26666666666666666}
\centering
\begin{tikzpicture}[line cap=round,line join=round,>=triangle 45,x=1.0cm,y=1.0cm,scale=0.6]
\fill[color=ffqqqq,fill=ffqqqq,fill opacity=0.05] (7.63288,-0.9404) -- (7.49978,6.5132) -- (14.021336949647692,2.90166798124371) -- cycle;
\draw [shift={(4.66,2.8)},fill=black,fill opacity=0.05] (0,0) -- (-51.52222804085102:0.6655) arc (-51.52222804085102:52.59197449391862:0.6655) -- cycle;
\fill[color=ffqqqq,fill=ffqqqq,fill opacity=0.05] (-1.76398,-0.1418) -- (-1.55102,6.5132) -- (-7.420899062185439,3.3701287699899334) -- cycle;
\draw [shift={(1.56352,3.02598)},fill=black,fill opacity=0.05] (0,0) -- (131.7689900996318:0.6655) arc (131.7689900996318:223.59137126466155:0.6655) -- cycle;
\draw [shift={(-1.73736,-8.76668)},line width=0.4pt,fill=black,fill opacity=0.15] (0,0) -- (-16.408628735338457:1.9965) arc (-16.408628735338457:11.768990099631793:1.9965) -- cycle;
\draw [color=qqqqff] (1.56352,3.02598)-- (4.66,2.8);
\draw [color=qqqqff] (4.66,2.8)-- (7.49978,6.5132);
\draw [color=qqqqff] (7.63288,-0.9404)-- (4.66,2.8);
\draw [line width=1.2pt,dash pattern=on 3pt off 3pt] (9.385721420315313,3.5043896024026813)-- (4.66,2.8);
\draw [line width=1.2pt,dash pattern=on 3pt off 3pt] (9.385721420315313,3.5043896024026826)-- (14.021336949647692,2.90166798124371);
\draw [color=ffqqqq] (7.63288,-0.9404)-- (7.49978,6.5132);
\draw [color=ffqqqq] (7.49978,6.5132)-- (14.021336949647692,2.90166798124371);
\draw [color=ffqqqq] (14.021336949647692,2.90166798124371)-- (7.63288,-0.9404);
\draw [color=qqqqff] (1.56352,3.02598)-- (-1.55102,6.5132);
\draw [color=qqqqff] (1.56352,3.02598)-- (-1.76398,-0.1418);
\draw [line width=1.2pt,dash pattern=on 3pt off 3pt] (-3.0137711085851695,2.0723192388972143)-- (-7.420899062185439,3.3701287699899334);
\draw [color=ffqqqq] (-1.76398,-0.1418)-- (-1.55102,6.5132);
\draw [color=ffqqqq] (-1.55102,6.5132)-- (-7.420899062185439,3.3701287699899334);
\draw [color=ffqqqq] (-7.420899062185439,3.3701287699899334)-- (-1.76398,-0.1418);
\draw [line width=1.2pt,dash pattern=on 3pt off 3pt] (-3.0137711085851695,2.0723192388972143)-- (1.56352,3.02598);
\draw (7.63288,-0.9404)-- (9.385721420315313,3.5043896024026813);
\draw (-3.0137711085851695,2.0723192388972143)-- (-1.55102,6.5132);
\draw [line width=1.2pt] (-2.9605311085851698,-4.2366207611027855)-- (-7.36765906218544,-2.9388112300100664);
\draw [line width=1.2pt] (-2.9605311085851698,-4.2366207611027855)-- (1.61676,-3.28296);
\draw [line width=1.2pt] (1.61676,-3.28296)-- (4.71324,-3.50894);
\draw [line width=1.2pt] (9.438961420315312,-2.8045503975973185)-- (4.71324,-3.50894);
\draw [line width=1.2pt] (9.438961420315312,-2.804550397597317)-- (14.074576949647692,-3.4072720187562897);
\draw(-1.73736,-8.76668) circle (1.9872503066297407cm);
\draw (2.66976795360027,-10.064489531092718)-- (-1.73736,-8.76668);
\draw (-1.73736,-8.76668)-- (2.8399311085851693,-7.813019238897214);
\draw (-1.73736,-8.76668)-- (1.35912,-8.99266);
\draw (2.9883614203153126,-8.062290397597318)-- (-1.73736,-8.76668);
\draw (-1.73736,-8.76668)-- (2.8982555293323795,-9.369401621158971);
\draw [line width=1.2pt,dotted] (-2.960531108585169,-4.236620761102785)-- (1.6750844207472113,-4.839342382261757);
\draw [line width=1.2pt,dotted] (1.6750844207472113,-4.839342382261757)-- (4.771564420747211,-5.065322382261757);
\draw [line width=1.2pt,dotted] (9.497285841062524,-4.360932779859075)-- (4.771564420747211,-5.065322382261757);
\draw [line width=1.2pt,dotted] (9.497285841062524,-4.360932779859075)-- (14.074576949647692,-3.4072720187562893);
\draw [color=xfqqff] (-7.36765906218544,-2.9388112300100664)-- (2.448598333829196,-5.82949934760827);
\draw [color=xfqqff] (2.448598333829196,-5.82949934760827)-- (14.074576949647692,-3.4072720187562897);
\draw [color=xfqqff] (-7.36765906218544,-2.9388112300100664)-- (14.074576949647692,-3.4072720187562897);

\fill [color=black] (1.56,3.03) circle (3pt);
\draw[color=black] (1.9,2.6) node {$s_1$};
\draw[color=black] (0.1,3.2) node {$\scriptstyle\frac{2\pi}{3}+\epsi_1$};
\draw[color=black] (2.1,3.6) node {$\scriptstyle\frac{2\pi}{3}+\epsi_2$};
\fill [color=black] (7.63,-0.94) circle (3pt);
\draw[color=black] (7.1,-0.9) node {$t_4$};
\fill [color=black] (7.5,6.51) circle (3pt);
\draw[color=black] (6.9,6.5) node {$t_3$};
\fill [color=black] (4.66,2.8) circle (3pt);
\draw[color=black] (4.4,3.2) node {$s_2$};
\draw[color=black] (4.2,2.2) node {$\scriptstyle\frac{2\pi}{3}+\epsi_3$};
\draw[color=black] (6.2,2.4) node {$\scriptstyle\frac{2\pi}{3}+\epsi_4$};
\draw[color=qqqqff] (3,2.5) node {$e_0$};
\draw[color=qqqqff] (5.6,5) node {$e_3$};
\draw[color=qqqqff] (5.5,1) node {$e_4$};
\fill [color=black] (9.39,3.5) circle (3pt);
\draw[color=black] (9.76,3.95) node {$t'_4$};
\fill [color=black] (4.66,2.8) circle (3pt);
\draw[color=black] (7,3.7) node {$e'_4$};
\fill [color=black] (14.02,2.9) circle (3pt);
\draw[color=black] (14.39,3.24) node {$t'_3$};
\fill [color=black] (9.39,3.5) circle (3pt);
\fill [color=black] (14.02,2.9) circle (3pt);
\draw[color=black] (12,2.6) node {$e'_3$};
\fill [color=black] (-1.55,6.51) circle (3pt);
\draw[color=black] (-1,6.6) node {$t_1$};
\draw[color=qqqqff] (0.5,5) node {$e_1$};
\fill [color=black] (-1.76,-0.14) circle (3pt);
\draw[color=black] (-1.1,-0.1) node {$t_2$};
\draw[color=qqqqff] (0.1,1) node {$e_2$};
\fill [color=black] (-3.01,2.07) circle (3pt);
\draw[color=black] (-2.64,1.6) node {$t'_1$};
\fill [color=black] (-7.42,3.37) circle (3pt);
\draw[color=black] (-7.7,2.9) node {$t'_2$};
\draw[color=black] (-4.8,3.1) node {$e'_2$};
\draw[color=black] (-0.6,2.1) node {$e'_1$};
\draw [fill=black] (-2.9605311085851698,-4.2366207611027855) circle (3pt);
\draw[color=black] (-2.8,-3.7) node {$t'_1$};
\draw[color=black] (-3,-4.6) node {$-\epsi_1$};
\draw [fill=black] (-7.36765906218544,-2.9388112300100664) circle (3pt);
\draw[color=black] (-7.1,-2.48436) node {$t'_2$};
\draw[color=black] (-4.5,-3.4) node {$e'_2$};
\draw [fill=black] (1.61676,-3.28296) circle (3pt);
\draw[color=black] (2.01606,-2.83042) node {$s'_1$};
\draw[color=black] (1.7,-3.7) node {$-\epsi_2$};
\draw[color=black] (-0.5,-3.3) node {$e'_1$};
\draw [fill=black] (4.71324,-3.50894) circle (3pt);
\draw[color=black] (4.9,-3.0) node {$s'_2$};
\draw[color=black] (4.85,-3.95) node {$\epsi_3$};
\draw[color=black] (3.4,-3.0) node {$e'_0$};
\draw [fill=black] (9.438961420315312,-2.8045503975973185) circle (3pt);
\draw[color=black] (9.84234,-2.35126) node {$t'_4$};
\draw[color=black] (9.35,-3.2) node {$\epsi_4$};
\draw[color=black] (7.28682,-2.6) node {$e'_4$};
\draw [fill=black] (14.074576949647692,-3.4072720187562897) circle (3pt);
\draw[color=black] (14.3,-2.96352) node {$t'_3$};
\draw[color=black] (11.9187,-2.7) node {$e'_3$};

\draw [fill=black] (2.66976795360027,-10.064489531092718) circle (3pt);
\draw[color=black] (3.16072,-10.2) node {$e'_2$};
\draw [fill=black] (-1.73736,-8.76668) circle (3pt);
\draw [fill=black] (2.8399311085851693,-7.813019238897214) circle (3pt);
\draw[color=black] (3.2,-7.4) node {$e'_1$};
\draw [fill=black] (-1.73736,-8.76668) circle (3pt);
\draw [fill=black] (1.35912,-8.99266) circle (3pt);
\draw[color=black] (1.82972,-8.8) node {$e'_0$};
\draw [fill=black] (2.9883614203153126,-8.062290397597318) circle (3pt);
\draw[color=black] (3.48016,-8.2) node {$e'_4$};
\draw [fill=black] (-1.73736,-8.76668) circle (3pt);
\draw [fill=black] (-1.73736,-8.76668) circle (3pt);
\draw [fill=black] (2.8982555293323795,-9.369401621158971) circle (3pt);
\draw[color=black] (3.37368,-9.3) node {$e'_3$};
\draw [fill=black] (-2.960531108585169,-4.236620761102785) circle (3pt);
\draw[color=black] (0.1,-4.2) node {$e_3'$};
\draw [fill=black] (1.6750844207472113,-4.839342382261757) circle (3pt);
\draw[color=black] (3.32044,-4.5) node {$e_0'$};
\draw [fill=black] (9.497285841062524,-4.360932779859075) circle (3pt);
\draw [fill=black] (4.771564420747211,-5.065322382261757) circle (3pt);
\draw[color=black] (7.18034,-4.2) node {$e_4'$};
\draw [fill=black] (9.497285841062524,-4.360932779859075) circle (3pt);
\draw[color=black] (11.99856,-4.2) node {$e_1'$};
\draw [fill=uuuuuu] (2.448598333829196,-5.82949934760827) circle (3pt);
\draw[color=uuuuuu] (2.62832,-5.4658) node {$p$};
\end{tikzpicture}
\caption{Unfolding an $\epsi$-approximate Steiner tree on $4$ terminals}\label{fig:4unfold}
\end{figure}
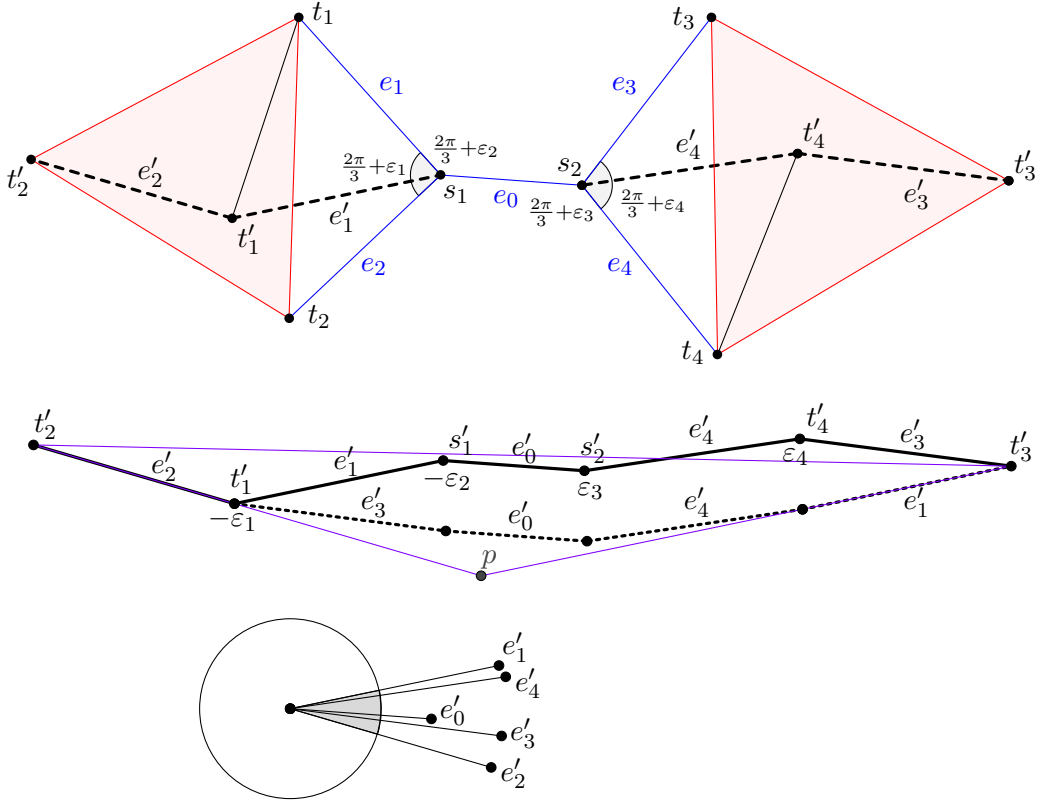
As in the proof of Proposition~\ref{prop:3unfold}, we unfold the tree into a polygonal line of total length $L(T)$, and with the distance between the endpoints a lower bound to $L(S(T))$.
Rotate $e_1$ by $\pi/3$ around $s_1$ to obtain $e_1'=s_1t_1'$.
Rotate $e_2$ by $-\pi/3$ around $t_1$ to obtain $e_2'=t_1't_2'$.
Rotate $e_3$ by $-\pi/3$ around $t_4$ to obtain $e_3'=t_4't_3'$.
Rotate $e_4$ by $\pi/3$ around $s_2$ to obtain $e_4'=s_2t_4'$.
This gives a polygonal line $P=t_2't_1's_1s_2t_4't_3'$ of length $L(T)$, with turns $-\epsi_1$ at $t_1'$, $-\epsi_2$ at $s_1$, $\epsi_3$ at $s_2$, and $\epsi_4$ at $t_4'$.
Note that the turn between any two of the five edges of $P$ will be at most $2\epsi$ in absolute value.
For instance, the absolute turn between $e_1'$ and $e_3'$ equals $\abs{-\epsi_2+\epsi_3+\epsi_4}\leq\abs{\epsi_2}+\abs{\epsi_3+\epsi_4}\leq2\epsi$.
If we reorder the edges of $P$ to make a new, convex polygonal line $P'$ with the same endpoints as $P$ (Fig.~\ref{fig:4unfold}, middle), then $P'$ will lie inside the triangle $\triangle t_2't_3'p$ bounded by $t_2't_3'$ and the lines through the first and last edges of $P'$.
The turn from the first edge to the last edge of $P'$ is exactly the maximum turn between two edges of $P$, so is at most $2\epsi$.
Hence, the angle at the apex of this triangle will be at least $\pi-2\epsi$, and by Lemma~\ref{lemma:reversetriangle}, $L(T)/\length{t_2't_3'}\leq 1/\cos\epsi$.
The proof of the upper bound concludes in the same way as that of Proposition~\ref{prop:3unfold}.

To show that $(\cos\epsi)^{-1}-1\geq G_2(\epsi,4)$, fix the above $\epsi$-approximate Steiner tree to have $\epsi_1=0$, $\epsi_2=\epsi$, $\epsi_3=-\epsi$, $\epsi_4=0$, $\length{s_1s_2}=\delta$ and $\length{s_1t_1}=\length{s_1t_2}=\length{s_2t_3}=\length{s_2t_4}=1$.
It is not difficult to see that the Melzak algorithm obtains a locally minimum Steiner tree $S(T)$ for any $\epsi<\pi/3$.
\end{proof}

The following generalises the idea in the above proof of estimating the length of a polygonal path in terms of the distance between its endpoints.
We do not know the history of this elementary result, but an extension of this lemma to curves of finite total curvature was proved by Schmidt \cite{Schmidt} (see also \cite[Theorem~5.8.1]{AR} and \cite[Proposition~7.1]{Sullivan}).
\begin{lemma}\label{schmidt}
Consider a planar polygonal path $p_0p_1\dots p_n$.
For each $i=1,\dots,n-1$, define the \define{turn} $\epsi_i$ at $p_i$ to be the signed angular measure in $[-\pi,\pi]$ by which the ray with source at $p_i$ in the direction opposite to $\Ray{p_ip_{i-1}}$ has to turn to coincide with the ray $\Ray{p_ip_{i+1}}$.
Let \[\kappa=\max_{1\leq i\leq j\leq n-1}\abs{\sum_{t=i}^j\epsi_t}.\]
If $\kappa<\pi$, then
\[ \frac{\sum\limits_{i=0}^{n-1}\length{p_i p_{i+1}}}{\length{p_0 p_n}}\leq\frac{1}{\cos(\kappa/2)}.\]
\end{lemma}
\begin{proof}
The case $n=2$ is just Lemma~\ref{lemma:reversetriangle}, so assume that $n\geq 3$.
Since $\kappa<\pi$, the $n$ unit vectors \[u_i=\norm{p_{i+1}-p_i}^{-1}(p_{i+1}-p_i)\] all lie in an open half circle.
The polygonal path $p_0p_1\dots p_n$ can be replaced with a convex polygonal path $p_o'p_1'\dots p_n'$ such that $p_0=p_o'$, $p_n=p_n'$ and each segment of the new path is a translation of a segment of the original path, selected so that the turns all have the same sign.
Then $p_0'p_1'\dots p_n'$ is a convex polygonal path with the same $\kappa$ and the same endpoints as the original polygonal path.
Let the lines $p_0'p_1'$ and $p_{n-1}'p_n'$ intersect in $q$.
Since $\kappa<\pi$, $p_0'p_1'\dots p_n'$ is contained in $\triangle p_o'qp_n'$.
By a well-known elementary geometric inequality, $\sum_{i=1}^{n-1}\length{p_i' p_{i+1}'}\leq\length{p_0'q}+\length{qp_n'}$.
It remains to apply the case $n=2$ of the lemma to the path $p_0'qp_n'$.
\end{proof}
\begin{proof}[Proof of Theorem~\ref{thm:plane}]
We choose a root edge of an $\epsi$-approximate Steiner tree $T$ on $n$ terminals and unfold the two parts of $T$ separated by the root edge to obtain a polygonal path $P$ with $2n-3$ edges, of the same length as $T$.
See Figure~\ref{fig:unfold}, where the blue $\epsi$-approximate tree has been unfolded.
The turn at each internal vertex of the polygonal path $P$ is indicated.
\begin{figure}
\centering
\definecolor{ffqqqq}{rgb}{1.,0.,0.}
\definecolor{yqqqyq}{rgb}{0.5019607843137255,0.,0.5019607843137255}
\definecolor{qqwuqq}{rgb}{0.,0.39215686274509803,0.}
\definecolor{zzttqq}{rgb}{0.6,0.2,0.}
\definecolor{ubqqys}{rgb}{0.29411764705882354,0.,0.5098039215686274}
\definecolor{qqqqff}{rgb}{0.,0.,1.}
\begin{tikzpicture}[thick,line cap=round,line join=round,>=triangle 45,scale=0.9]
\clip(0.9572076338717613,-3.8097111621500117) rectangle (16.170091294260665,6.077246747488182);
\draw [shift={(4.5408757587119695,0.7088269083006826)},color=qqwuqq,fill=qqwuqq,fill opacity=0.1] (0,0) -- (160.8389474728972:0.4249408843684051) arc (160.8389474728972:242.47884704013373:0.4249408843684051) -- cycle;
\draw [shift={(4.5408757587119695,0.7088269083006826)},color=yqqqyq,fill=yqqqyq,fill opacity=0.1] (0,0) -- (23.11420798332681:0.4249408843684051) arc (23.11420798332681:160.8389474728972:0.4249408843684051) -- cycle;
\draw [shift={(5.702380842652275,1.204591273397154)},color=qqwuqq,fill=qqwuqq,fill opacity=0.1] (0,0) -- (-156.8857920166732:0.4249408843684051) arc (-156.8857920166732:-59.478429238377394:0.4249408843684051) -- cycle;
\draw [shift={(6.82139183815574,-0.693478010115051)},color=qqwuqq,fill=qqwuqq,fill opacity=0.1] (0,0) -- (120.52157076162264:0.4249408843684051) arc (120.52157076162264:227.6701112296518:0.4249408843684051) -- cycle;
\draw [shift={(6.82139183815574,-0.693478010115051)},color=ubqqys,fill=ubqqys,fill opacity=0.1] (0,0) -- (-132.32988877034825:0.4249408843684051) arc (-132.32988877034825:17.85532014033951:0.4249408843684051) -- cycle;
\draw [shift={(6.82139183815574,-0.693478010115051)},color=ffqqqq,fill=ffqqqq,fill opacity=0.1] (0,0) -- (17.855320140339497:0.4249408843684051) arc (17.855320140339497:120.52157076162264:0.4249408843684051) -- cycle;
\draw [shift={(5.702380842652275,1.204591273397154)},color=yqqqyq,fill=yqqqyq,fill opacity=0.1] (0,0) -- (-59.478429238377394:0.4249408843684051) arc (-59.478429238377394:57.97354577543265:0.4249408843684051) -- cycle;
\draw [shift={(6.4834,2.4532)},color=qqwuqq,fill=qqwuqq,fill opacity=0.1] (0,0) -- (-122.02645422456737:0.4249408843684051) arc (-122.02645422456737:35.78803442955335:0.4249408843684051) -- cycle;
\draw [shift={(6.4834,2.4532)},color=yqqqyq,fill=yqqqyq,fill opacity=0.1] (0,0) -- (35.78803442955335:0.4249408843684051) arc (35.78803442955335:115.47254322022175:0.4249408843684051) -- cycle;
\draw [color=qqqqff] (5.702380842652275,1.204591273397154)-- (6.4834,2.4532);
\draw [color=qqqqff] (5.702380842652275,1.204591273397154)-- (4.5408757587119695,0.7088269083006826);
\draw [color=qqqqff] (5.702380842652275,1.204591273397154)-- (6.82139183815574,-0.693478010115051);
\draw [color=qqqqff] (6.4834,2.4532)-- (8.209532060425861,3.6975777950251247);
\draw [color=qqqqff] (6.4834,2.4532)-- (5.8298631079627965,3.825060060335646);
\draw [color=qqqqff] (4.5408757587119695,0.7088269083006826)-- (3.3,1.14);
\draw [color=qqqqff] (4.5408757587119695,0.7088269083006826)-- (3.4620185849627187,-1.3617716589397337);
\draw [color=qqqqff] (6.82139183815574,-0.693478010115051)-- (5.60836004837736,-2.0251862930860765);
\draw [color=qqqqff] (6.82139183815574,-0.693478010115051)-- (8.320555170328587,-0.21055214674460948);
\draw [color=zzttqq] (4.5408757587119695,0.7088269083006826)-- (3.5470310285160984,-0.15021647583451392);
\draw [color=zzttqq] (3.547031028516099,-0.15021647583451392)-- (1.2144114813716116,-0.2511980399327891);
\draw [color=zzttqq] (6.4834,2.4532)-- (6.2688032478159075,4.570263112128134);
\draw [color=zzttqq] (6.268803247815907,4.570263112128134)-- (7.130100464485228,5.8221726931105025);
\draw [color=zzttqq] (6.82139183815574,-0.693478010115051)-- (7.368169146749613,-2.4098484971467453);
\draw [color=zzttqq] (7.368169146749613,-2.409848497146745)-- (8.535976878659373,-3.466699095545339);
\draw [color=zzttqq] (5.702380842652275,1.204591273397154)-- (7.905662558068505,1.224648580861166);
\draw [color=zzttqq] (7.905662558068505,1.224648580861166)-- (9.665471656440758,0.8399863768004967);
\draw [color=zzttqq] (9.665471656440758,0.8399863768004969)-- (11.164634988613606,1.3229122401709383);
\draw [color=zzttqq] (11.164634988613606,1.3229122401709388)-- (12.636471443912471,1.270834172366917);
\draw [color=zzttqq] (12.636471443912471,1.270834172366917)-- (14.362603504338333,2.515211967392042);
\draw [color=zzttqq] (14.362603504338333,2.515211967392042)-- (15.877437613044858,2.3952614880387655);
\begin{scriptsize}
\draw [fill=black] (4.5408757587119695,0.7088269083006826) circle (1.5pt);
\draw [fill=black] (4.8,0.5) node {$-\epsi_2$};
\draw [fill=black] (8.209532060425861,3.6975777950251247) circle (1.5pt);
\draw [fill=black] (5.8298631079627965,3.825060060335646) circle (1.5pt);
\draw [fill=black] (3.3,1.14) circle (1.5pt);
\draw [fill=black] (3.4620185849627187,-1.3617716589397337) circle (1.5pt);
\draw [fill=black] (8.320555170328587,-0.21055214674460948) circle (1.5pt);
\draw [fill=black] (5.60836004837736,-2.0251862930860765) circle (1.5pt);
\draw[color=qqqqff] (4.3,-0.3) node {$a$};
\draw[color=qqqqff] (3.7,1.3) node {$b$};
\draw[color=qqqqff] (5.1,1.15) node {$c$};
\draw[color=qqqqff] (6.1,0) node {$d$};
\draw[color=qqqqff] (6.0,-1.2) node {$e$};
\draw[color=qqqqff] (7.713767695329403,-0.7) node {$f$};
\draw[color=qqqqff] (5.9,1.9) node {$g$};
\draw[color=qqqqff] (7.5,2.9) node {$h$};
\draw[color=qqqqff] (5.9,3.2) node {$i$};
\draw[color=zzttqq] (4,0.5) node {$b$};
\draw [fill=black] (3.547031028516099,-0.15021647583451392) circle (1.5pt);
\draw [fill=black] (3.5,-0.4) node {$-\epsi_1$};
\draw [fill=black] (1.2144114813716116,-0.2511980399327891) circle (1.5pt);
\draw[color=zzttqq] (2.5,-0.4) node {$a$};
\draw [fill=black] (6.4834,2.4532) circle (1.5pt);
\draw(6.2,2.45) node {$\epsi_7$};
\draw[color=zzttqq] (6.55,3.9) node {$h$};
\draw [fill=black] (6.268803247815907,4.570263112128134) circle (1.5pt);
\draw(6.2,4.9) node {$\epsi_8$};
\draw [fill=black] (7.130100464485228,5.8221726931105025) circle (1.5pt);
\draw[color=zzttqq] (7.0055328880487275,5.1) node {$i$};
\draw [fill=black] (6.82139183815574,-0.693478010115051) circle (1.5pt);
\draw (6.95,-0.45) node {$\epsi_4$};
\draw[color=zzttqq] (6.8,-1.6) node {$e$};
\draw [fill=black] (7.368169146749613,-2.409848497146745) circle (1.5pt);
\draw (7.3,-2.65) node {$\epsi_5$};
\draw [fill=black] (8.535976878659373,-3.466699095545339) circle (1.5pt);
\draw[color=zzttqq] (7.7,-3.1) node {$f$};
\draw [fill=black] (5.702380842652275,1.204591273397154) circle (1.5pt);
\draw [fill=black] (5.4,1.3) node {$\epsi_3$};
\draw [fill=black] (7.905662558068505,1.224648580861166) circle (1.5pt);
\draw [fill=black] (7.9,1.5) node {$\epsi_4$};
\draw[color=zzttqq] (6.835556534301365,1) node {$d$};
\draw [fill=black] (9.665471656440758,0.8399863768004967) circle (1.5pt);
\draw [fill=black] (9.8,0.55) node {$\epsi_5$};
\draw[color=zzttqq] (8.776119906250415,0.8) node {$e$};
\draw [fill=black] (11.164634988613606,1.3229122401709383) circle (1.5pt);
\draw[color=zzttqq] (10.518377532160875,0.9) node {$f$};
\draw [fill=black] (11.2,1.6) node {$-\epsi_4-\epsi_5-\epsi_6$};
\draw [fill=black] (12.636471443912471,1.270834172366917) circle (1.5pt);
\draw [fill=black] (13.0,1.2) node {$\epsi_7$};
\draw[color=zzttqq] (11.8,1.05) node {$g$};
\draw [fill=black] (14.362603504338333,2.515211967392042) circle (1.5pt);
\draw [fill=black] (14.362603504338333,2.8) node {$\epsi_8$};
\draw[color=zzttqq] (13.662940076487073,1.7) node {$h$};
\draw [fill=black] (15.877437613044858,2.3952614880387655) circle (1.5pt);
\draw[color=zzttqq] (15.136068475630879,2.7) node {$i$};
\end{scriptsize}
{\tiny
\draw[color=qqwuqq] (3.3,0.6) node {$\frac{2\pi}{3}+\epsi_1$};
\draw[color=yqqqyq] (4.5,1.4) node {$\frac{2\pi}{3}+\epsi_2$};
\draw[color=qqwuqq] (5.5,0.6) node {$\frac{2\pi}{3}+\epsi_3$};
\draw[color=qqwuqq] (5.8,-0.6226545293869838) node {$\frac{2\pi}{3}+\epsi_4$};
\draw[color=ubqqys] (7.75,-1.25) node {$\frac{2\pi}{3}+\epsi_5$};
\draw[color=ffqqqq] (7.3,0.1) node {$\frac{2\pi}{3}-\epsi_4-\epsi_5$};
\draw[color=yqqqyq] (6.7,1.6) node {$\frac{2\pi}{3}+\epsi_6$};
\draw[color=qqwuqq] (7.6,2.2) node {$\frac{2\pi}{3}+\epsi_7$};
\draw[color=yqqqyq] (6.7,3.2) node {$\frac{2\pi}{3}+\epsi_8$};
}
\end{tikzpicture}
\caption{Proof of Theorem~\ref{thm:plane}}\label{fig:unfold}
\end{figure}
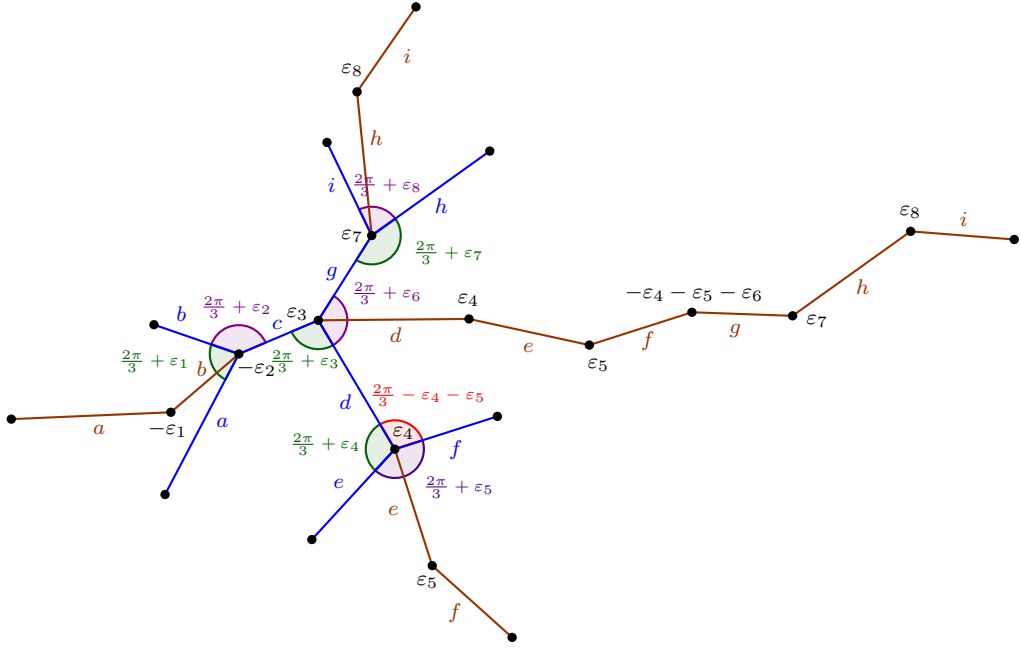
The quantity $\kappa$ of Lemma~\ref{schmidt} is the maximum absolute turn between any two edges of $P$.
For example, the total turn between edge $a$ and edge $h$ on $P$ in Fig.~\ref{fig:unfold} equals $-\epsi_1-\epsi_2+\epsi_3+\epsi_4+\epsi_5-\epsi_4-\epsi_5-\epsi_6 = (-\epsi_1-\epsi_2) -\epsi_6$, which is the sum of the errors at the two Steiner points on the path between edges $a$ and $h$ in the tree.
Thus, the absolute turn between $a$ and $h$ in $P$ is at most $2\epsi$.
In general, since there are at most $n-2$ Steiner points in a full Steiner topology on $n$ terminals, there are at most $n-2$ Steiner points on the path between any two edges in an $\epsi$-approximate Steiner tree, each contributing an error of absolute value at most $\epsi$.
It follows that $\kappa\leq(n-2)\epsi$.
We now apply Lemma~\ref{schmidt} to obtain that $L(T)/L(S(T))\leq 1/\cos(\frac12(n-2)\epsi)$.
\end{proof}

\section{\texorpdfstring{Construction of an $\epsi$-Approximate Full Binary Tree in the Plane}{Construction of an Epsilon-Approximate Full Binary Tree in the Plane}}\label{section:construction}
In this section we prove Theorems~\ref{theorem:lowerbound} and \ref{theorem:lowerbound'} by constructing a sequence of $\epsi$-approximate Steiner trees $T_k$ ($k\in\bN$) for which it is possible to calculate the ratio between their length and the length of a locally minimum Steiner tree on the same terminals, if $\epsi\leq 1/k^2$.
A somewhat similar construction is made in \cite{PST}.
The calculation will make essential use of complex numbers.
Using complex numbers to solve problems in classical Euclidean geometry is an old trick \cite{AA, Hahn, Schwerdtfeger}, and even in the geometric Steiner tree literature there are papers where complex numbers appear \cite{Biggs, Booth}.
\begin{proof}[Proof of Theorems~\ref{theorem:lowerbound} and \ref{theorem:lowerbound'}]
Throughout the proof we denote the largest integer not greater than $x$ by $\lfloor x\rfloor$.
Fix $k\in\bN$.
We describe an $\epsi$-approximate Steiner tree $T_k$ with $2^k+1$ terminals $p_i$ (for $i=0$ and $2^k\leq i\leq 2^{k+1}-1$), $2^k-1$ Steiner points $p_i$ ($1\leq i\leq 2^k-1$) and $2^{k+1}-1$ edges $e_i=p_ip_{\lfloor i/2\rfloor}$ ($1\leq i\leq 2^{k+1}-1$).
Let each $e_i$ have length $2^{-\lfloor\log_2 i\rfloor}$, and let the angles at the edges incident to the Steiner point $p_i$ be $\myangle p_{2i}p_ip_{2i+1}=2\pi/3$, $\myangle p_{2i+1}p_ip_{\lfloor i/2\rfloor}=2\pi/3-\epsi$, and $\myangle p_{\lfloor i/2\rfloor}p_i p_{2i}=2\pi/3+\epsi$ (Figure~\ref{angles}).
\begin{figure}
\centering
\begin{tikzpicture}[scale=2]
\draw (0,0) -- (180:1cm);
\fill (180:1cm) circle (1pt) node[left]{$p_{\lfloor i/2\rfloor}$};
\draw (0,0) -- (70:0.5cm);
\fill (70:0.5cm) circle (1pt) node[above right]{$p_{2i+1}$};
\draw (0,0) -- (-50:0.5cm);
\fill (-50:0.5cm) circle (1pt) node[below right]{$p_{2i}$};
\fill (0,0) circle (1pt) node[right=0.5mm]{$p_i$} node[below left=0mm] {$\scriptscriptstyle2\pi/3+\epsi$} node[above left=0mm] {$\scriptscriptstyle2\pi/3-\epsi$} node[above right=1mm] {$\scriptscriptstyle2\pi/3$};
\end{tikzpicture}
\caption{The angles around a Steiner point in the binary tree construction}\label{angles}
\end{figure}
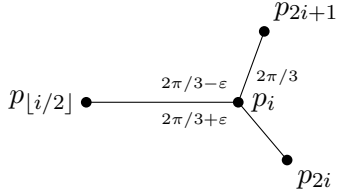
This determines the tree uniquely up to congruence.
See Figure~\ref{k=3} for the case $k=3$.
\begin{figure}
\centering
\definecolor{ffqqqq}{rgb}{1,0,0}
\definecolor{yqyqyq}{rgb}{0.5,0.5,0.5}
\definecolor{qqwuqq}{rgb}{0,0.39,0}
\definecolor{qqqqff}{rgb}{0,0,1}
\definecolor{xdxdff}{rgb}{0.49,0.49,1}
\definecolor{uuuuuu}{rgb}{0.27,0.27,0.27}
\begin{tikzpicture}[scale=0.9,line cap=round,line join=round,>=triangle 45,x=5.0cm,y=5.0cm]
\clip(0.43,-4.98) rectangle (3.56,-0.86);
\draw [shift={(2.28,-3.92)},color=qqwuqq,fill=qqwuqq,fill opacity=0.1] (0,0) -- (36.6:0.09) arc (36.6:156.6:0.09) -- cycle;
\draw [shift={(2.68,-3.62)},color=qqwuqq,fill=qqwuqq,fill opacity=0.1] (0,0) -- (-17:0.09) arc (-17:103:0.09) -- cycle;
\draw [shift={(1.82,-3.72)},color=qqwuqq,fill=qqwuqq,fill opacity=0.1] (0,0) -- (103:0.09) arc (103:223:0.09) -- cycle;
\draw [shift={(2.92,-3.69)},color=qqwuqq,fill=qqwuqq,fill opacity=0.1] (0,0) -- (-70.6:0.09) arc (-70.6:49.4:0.09) -- cycle;
\draw [shift={(2.62,-3.38)},color=qqwuqq,fill=qqwuqq,fill opacity=0.1] (0,0) -- (49.4:0.09) arc (49.4:169.4:0.09) -- cycle;
\draw [shift={(1.76,-3.48)},color=qqwuqq,fill=qqwuqq,fill opacity=0.1] (0,0) -- (49.4:0.09) arc (49.4:169.4:0.09) -- cycle;
\draw [shift={(1.63,-3.89)},color=qqwuqq,fill=qqwuqq,fill opacity=0.1] (0,0) -- (169.4:0.09) arc (169.4:289.4:0.09) -- cycle;
\fill[line width=1.2pt,color=yqyqyq,fill=yqyqyq,fill opacity=0.1] (1.64,-3.45) -- (1.84,-3.38) -- (1.68,-3.24) -- cycle;
\fill[line width=1.2pt,color=yqyqyq,fill=yqyqyq,fill opacity=0.1] (1.68,-4.01) -- (1.51,-3.87) -- (1.47,-4.08) -- cycle;
\fill[line width=1.2pt,color=yqyqyq,fill=yqyqyq,fill opacity=0.1] (2.5,-3.36) -- (2.7,-3.28) -- (2.54,-3.14) -- cycle;
\fill[line width=1.2pt,color=yqyqyq,fill=yqyqyq,fill opacity=0.1] (3,-3.6) -- (2.96,-3.81) -- (3.16,-3.74) -- cycle;
\fill[line width=1.2pt,color=yqyqyq,fill=yqyqyq,fill opacity=0.1] (1.47,-4.08) -- (1.68,-3.24) -- (0.85,-3.48) -- cycle;
\fill[line width=1.2pt,color=yqyqyq,fill=yqyqyq,fill opacity=0.1] (2.54,-3.14) -- (3.16,-3.74) -- (3.37,-2.9) -- cycle;
\fill[line width=1.2pt,color=yqyqyq,fill=yqyqyq,fill opacity=0.1] (0.85,-3.48) -- (3.37,-2.9) -- (1.6,-1.01) -- cycle;
\draw [line width=2pt,color=qqqqff] (2.28,-4.92)-- (2.28,-3.92);
\draw [line width=1.6pt,color=qqqqff] (1.82,-3.72)-- (2.28,-3.92)-- (2.68,-3.62);
\draw [line width=1.6pt,color=qqqqff] (2.62,-3.38)-- (2.68,-3.62)-- (2.92,-3.69);
\draw [line width=1.6pt,color=qqqqff] (1.63,-3.89)-- (1.82,-3.72)-- (1.76,-3.48);
\draw [line width=1.6pt,color=qqqqff] (2.96,-3.81)-- (2.92,-3.69)-- (3,-3.6);
\draw [line width=1.6pt,color=qqqqff] (2.7,-3.28)-- (2.62,-3.38)-- (2.5,-3.36);
\draw [line width=1.6pt,color=qqqqff] (1.84,-3.38)-- (1.76,-3.48)-- (1.64,-3.45);
\draw [line width=1.6pt,color=qqqqff] (1.51,-3.87)-- (1.63,-3.89)-- (1.68,-4.01);
\draw [line width=1.2pt,color=yqyqyq] (1.64,-3.45)-- (1.84,-3.38);
\draw [line width=1.2pt,color=yqyqyq] (1.84,-3.38)-- (1.68,-3.24);
\draw [line width=1.2pt,color=yqyqyq] (1.68,-3.24)-- (1.64,-3.45);
\draw [line width=1.2pt,color=yqyqyq] (1.68,-4.01)-- (1.51,-3.87);
\draw [line width=1.2pt,color=yqyqyq] (1.51,-3.87)-- (1.47,-4.08);
\draw [line width=1.2pt,color=yqyqyq] (1.47,-4.08)-- (1.68,-4.01);
\draw [line width=1.2pt,color=yqyqyq] (2.5,-3.36)-- (2.7,-3.28);
\draw [line width=1.2pt,color=yqyqyq] (2.7,-3.28)-- (2.54,-3.14);
\draw [line width=1.2pt,color=yqyqyq] (2.54,-3.14)-- (2.5,-3.36);
\draw [line width=1.2pt,color=yqyqyq] (3,-3.6)-- (2.96,-3.81);
\draw [line width=1.2pt,color=yqyqyq] (2.96,-3.81)-- (3.16,-3.74);
\draw [line width=1.2pt,color=yqyqyq] (3.16,-3.74)-- (3,-3.6);
\draw(1.72,-3.36) circle (0.62cm);
\draw(2.58,-3.26) circle (0.62cm);
\draw(1.55,-3.99) circle (0.63cm);
\draw(3.04,-3.72) circle (0.62cm);
\draw [line width=1.2pt,color=yqyqyq] (1.47,-4.08)-- (1.68,-3.24);
\draw [line width=1.2pt,color=yqyqyq] (1.68,-3.24)-- (0.85,-3.48);
\draw [line width=1.2pt,color=yqyqyq] (0.85,-3.48)-- (1.47,-4.08);
\draw [line width=1.2pt,color=yqyqyq] (2.54,-3.14)-- (3.16,-3.74);
\draw [line width=1.2pt,color=yqyqyq] (3.16,-3.74)-- (3.37,-2.9);
\draw [line width=1.2pt,color=yqyqyq] (3.37,-2.9)-- (2.54,-3.14);
\draw [line width=1.2pt,color=yqyqyq] (0.85,-3.48)-- (3.37,-2.9);
\draw [line width=1.2pt,color=yqyqyq] (3.37,-2.9)-- (1.6,-1.01);
\draw [line width=1.2pt,color=yqyqyq] (1.6,-1.01)-- (0.85,-3.48);
\draw(1.33,-3.6) circle (2.5cm);
\draw(3.02,-3.26) circle (2.5cm);
\draw(1.94,-2.46) circle (7.47cm);
\draw (1.6,-1.01)-- (2.28,-4.92);
\draw (1.6,-1.01)-- (2.28,-3.92);
\draw [line width=1.6pt,color=ffqqqq] (2.11,-3.95)-- (2.28,-4.92);
\draw (2.11,-3.95)-- (0.85,-3.48);
\draw (2.11,-3.95)-- (3.37,-2.9);
\draw [line width=1.6pt,color=ffqqqq] (1.78,-3.83)-- (2.11,-3.95);
\draw [line width=1.6pt,color=ffqqqq] (2.11,-3.95)-- (2.61,-3.54);
\draw (1.78,-3.83)-- (1.47,-4.08);
\draw (1.78,-3.83)-- (1.68,-3.24);
\draw (2.61,-3.54)-- (2.54,-3.14);
\draw [color=uuuuuu] (2.61,-3.54)-- (3.16,-3.74);
\draw [line width=1.6pt,color=ffqqqq] (1.66,-3.92)-- (1.78,-3.83);
\draw [line width=1.6pt,color=ffqqqq] (1.78,-3.83)-- (1.72,-3.48);
\draw [line width=1.6pt,color=ffqqqq] (1.72,-3.48)-- (1.64,-3.45);
\draw [line width=1.6pt,color=ffqqqq] (1.72,-3.48)-- (1.84,-3.38);
\draw [line width=1.6pt,color=ffqqqq] (2.61,-3.54)-- (2.58,-3.39);
\draw [line width=1.6pt,color=ffqqqq] (2.58,-3.39)-- (2.5,-3.36);
\draw [line width=1.6pt,color=ffqqqq] (2.58,-3.39)-- (2.7,-3.28);
\draw [line width=1.6pt,color=ffqqqq] (2.93,-3.66)-- (3,-3.6);
\draw [line width=1.6pt,color=ffqqqq] (2.93,-3.66)-- (2.96,-3.81);
\draw [line width=1.6pt,color=ffqqqq] (2.61,-3.54)-- (2.93,-3.66);
\draw [line width=1.6pt,color=ffqqqq] (1.66,-3.92)-- (1.51,-3.87);
\draw [line width=1.6pt,color=ffqqqq] (1.66,-3.92)-- (1.68,-4.01);
\begin{scriptsize}
\fill [color=qqqqff] (2.28,-4.92) circle (1.5pt);
\draw [color=black] (2.28,-4.92) circle (1.5pt);
\draw (2.28,-4.92) node[right] {$p_0$};
\fill [color=qqqqff] (2.28,-3.92) circle (1.5pt);
\draw [color=black] (2.28,-3.92) circle (1.5pt);
\draw (2.35,-3.85) node[below=2mm] {$p_1$};
\fill [color=qqqqff] (2.68,-3.62) circle (1.5pt);
\draw [color=black] (2.68,-3.62) circle (1.5pt);
\draw (2.69,-3.62) node [below] {$p_2$};
\fill [color=qqqqff] (1.82,-3.72) circle (1.5pt);
\draw [color=black] (1.82,-3.72) circle (1.5pt);
\draw (1.84,-3.72) node[below] {$p_3$};
\fill [color=qqqqff] (2.92,-3.69) circle (1.5pt);
\draw [color=black] (2.92,-3.69) circle (1.5pt);
\draw (2.93,-3.71) node[left] {$p_4$};
\fill [color=qqqqff] (2.62,-3.38) circle (1.5pt);
\draw [color=black] (2.62,-3.38) circle (1.5pt);
\draw (2.62,-3.38) node[right] {$p_5$};
\fill [color=qqqqff] (1.76,-3.48) circle (1.5pt);
\draw [color=black] (1.76,-3.48) circle (1.5pt);
\draw (1.76,-3.50) node[right] {$p_6$};
\fill [color=qqqqff] (1.63,-3.89) circle (1.5pt);
\draw [color=black] (1.63,-3.89) circle (1.5pt);
\draw (1.63,-3.89) node[above] {$p_7$};
\fill [color=black] (2.96,-3.81) circle (1.5pt);
\draw [color=black] (2.96,-3.81) circle (1.5pt);
\draw (2.96,-3.81) node[below] {$p_8$};
\fill [color=black] (3,-3.6) circle (1.5pt);
\draw [color=black] (3,-3.6) circle (1.5pt);
\draw (3.03,-3.6) node[above] {$p_9$};
\fill [color=black] (2.7,-3.28) circle (1.5pt);
\draw [color=black] (2.7,-3.28) circle (1.5pt);
\draw (2.7,-3.28) node[right] {$p_{10}$};
\fill [color=black] (2.5,-3.36) circle (1.5pt);
\draw [color=black] (2.5,-3.36) circle (1.5pt);
\draw (2.51,-3.36) node[left] {$p_{11}$};
\fill [color=black] (1.84,-3.38) circle (1.5pt);
\draw [color=black] (1.84,-3.38) circle (1.5pt);
\draw (1.84,-3.38) node[right] {$p_{12}$};
\fill [color=black] (1.64,-3.45) circle (1.5pt);
\draw [color=black] (1.64,-3.45) circle (1.5pt);
\draw (1.64,-3.45) node[left] {$p_{13}$};
\fill [color=black] (1.51,-3.87) circle (1.5pt);
\draw [color=black] (1.51,-3.87) circle (1.5pt);
\draw (1.48,-3.87) node[above] {$p_{14}$};
\fill [color=black] (1.68,-4.01) circle (1.5pt);
\draw [color=black] (1.68,-4.01) circle (1.5pt);
\draw (1.68,-4.01) node[right] {$p_{15}$};
\fill [color=black] (1.68,-3.24) circle (1.5pt);
\draw [color=black] (1.68,-3.24) circle (1.5pt);
\draw (1.68,-3.24) node[above] {$q_6$};
\fill [color=black] (1.47,-4.08) circle (1.5pt);
\draw [color=black] (1.47,-4.08) circle (1.5pt);
\draw (1.47,-4.08) node[below] {$q_7$};
\fill [color=black] (2.54,-3.14) circle (1.5pt);
\draw [color=black] (2.54,-3.14) circle (1.5pt);
\draw (2.54,-3.14) node[left] {$q_5$};
\fill [color=black] (3.16,-3.74) circle (1.5pt);
\draw [color=black] (3.16,-3.74) circle (1.5pt);
\draw (3.15,-3.76) node[right] {$q_4$};
\fill [color=black] (0.85,-3.48) circle (1.5pt);
\draw [color=black] (0.85,-3.48) circle (1.5pt);
\draw (0.85,-3.48) node[left] {$q_3$};
\fill [color=black] (3.37,-2.9) circle (1.5pt);
\draw [color=black] (3.37,-2.9) circle (1.5pt);
\draw (3.37,-2.9) node[right] {$q_2$};
\fill [color=black] (1.6,-1.01) circle (1.5pt);
\draw [color=black] (1.6,-1.01) circle (1.5pt);
\draw (1.6,-1.01) node[above] {$q_1$};
\fill [color=ffqqqq] (2.11,-3.95) circle (1.5pt);
\draw [color=black] (2.11,-3.95) circle (1.5pt);
\draw (2.13,-3.95) node[below left] {$s_1$};
\fill [color=ffqqqq] (1.78,-3.83) circle (1.5pt);
\draw [color=black] (1.78,-3.83) circle (1.5pt);
\draw (1.79,-3.83) node[below=0mm] {$s_3$};
\fill [color=ffqqqq] (2.61,-3.54) circle (1.5pt);
\draw [color=black] (2.61,-3.54) circle (1.5pt);
\draw (2.61,-3.53) node[left] {$s_2$};
\fill [color=ffqqqq] (1.66,-3.92) circle (1.5pt);
\draw [color=black] (1.66,-3.92) circle (1.5pt);
\draw (1.66,-3.92) node[right] {$s_7$};
\fill [color=ffqqqq] (1.72,-3.48) circle (1.5pt);
\draw [color=black] (1.72,-3.48) circle (1.5pt);
\draw (1.74,-3.50) node[left] {$s_6$};
\fill [color=ffqqqq] (2.58,-3.39) circle (1.5pt);
\draw [color=black] (2.58,-3.39) circle (1.5pt);
\draw[color=black] (2.60,-3.41) node[left] {$s_5$};
\fill [color=ffqqqq] (2.93,-3.66) circle (1.5pt);
\draw [color=black] (2.93,-3.66) circle (1.5pt);
\draw (2.92,-3.665) node[above] {$s_4$};
\end{scriptsize}
\end{tikzpicture}
\caption{Construction of $T_k$, $k=3$}\label{k=3}
\end{figure}
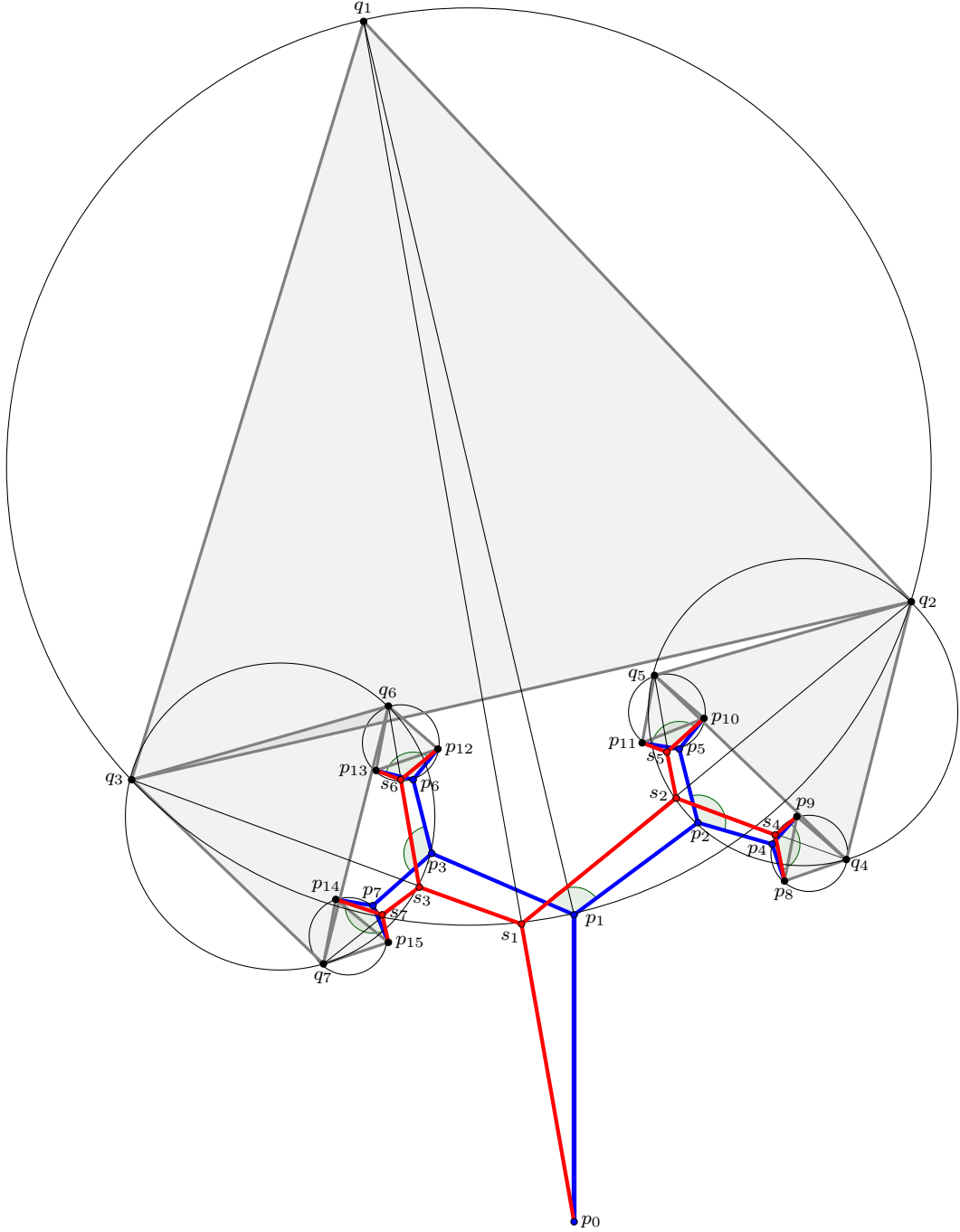
Since there are $2^j$ edges of length $2^{-j+1}$ ($j=0,1,\dots,k$),
\begin{equation}\label{lengthTk}
L(T_k)=k+1.
\end{equation}
We construct this tree recursively, using complex numbers.
Let $p_0=0\in\bC$ and $p_1=1\in\bC$.
Then $e_1=p_0p_1$.
Let $\omega=e^{i\pi/3}$ and $z=e^{i\epsi}$.

Once $p_{\lfloor i/2\rfloor}$ and $p_i$ have been defined, define $p_{2i}$ and $p_{2i+1}$ as in Figure~\ref{angles}.
If we walk from $p_{\lfloor i/2\rfloor}$ to $p_i$ and then turn in the direction of $p_{2i}$, the turn is a right turn by an angle of $\pi/3-\epsi$.
Also, $\length{p_ip_{2i}}=\frac12 \length{p_{\lfloor i/2\rfloor}p_i}$.
Therefore,
\begin{equation}\label{*}
p_{2i}-p_i = \frac12(p_i-p_{\lfloor i/2\rfloor})\omega^{-1}z.
\end{equation}
Similarly, if we turn instead in the direction of $p_{2i+1}$, this is a left turn by an angle of $\pi/3+\epsi$, which gives
\begin{equation}\label{***}
p_{2i+1}-p_i = \frac12(p_i-p_{\lfloor i/2\rfloor})\omega z.
\end{equation}
We obtain the following recurrence:
\begin{equation}\label{**}
\left.\begin{aligned}
p_0 &=0, \quad p_1=1,\\
p_{2i} &= p_i+\frac12(p_i-p_{\lfloor i/2\rfloor})\omega^{-1}z,\quad i\geq 1\\
p_{2i+1} &= p_i+\frac12(p_i-p_{\lfloor i/2\rfloor})\omega z,\quad i\geq 1.\\
\end{aligned}\right\}
\end{equation}
To describe its solution, we have to consider the sequence of left and right turns as we walk from $p_0$ to $p_i$.
This can be found from the binary expression of $i$.
Let $h(i)=\lfloor\log_2 i\rfloor$.
Let $b_0,b_1,\dots,b_{h(i)}\in\{0,1\}$ be the unique values such that \[i=\sum_{j=0}^{h(i)-1} b_j 2^j + 2^{h(i)}.\]
If we replace $0$ by $R$ and $1$ by $L$ in the sequence $b_{h(i)-1},\dots, b_0$, we obtain the left and right turns in the path from $p_0$ to $p_i$.
Let $a_j(i)$ be the number of $1$s in $b_{h(i)-1}, \cdots, b_{h(i)-j}$ minus the number of $0$s in $b_{h(i)-1}, \cdots, b_{h(i)-j}$.
In particular, $a_0(i)=0$.
\begin{lemma}
For each $i\geq 1$,
\begin{equation}\label{10}
p_i=\sum_{j=0}^{h(i)} \omega^{a_j(i)}\left(\frac{z}{2}\right)^j.
\end{equation}
\end{lemma}
\begin{proof}
Observe that $h(2i)=h(2i+1)=h(i)+1$, 
\[ a_j(2i)=a_j(2i+1)=a_j(i)\text{ for each $j=0,\dots,h(i)$, and}\]
\begin{equation}\label{+}
a_{h(i)}(i)=a_{h(2i)}(2i)+1=a_{h(2i)-1}(2i)=a_{h(2i+1)}(2i+1)-1=a_{h(2i+1)-1}(2i+1).
\end{equation}
It then follows by induction, using \eqref{*} and \eqref{***}, that
\begin{equation}\label{4}
p_i-p_{\lfloor i/2\rfloor}=\omega^{a_{h(i)}(i)}\left(\frac{z}{2}\right)^{h(i)}.
\end{equation}
Finally, by induction and \eqref{**} we obtain \eqref{10}.
\end{proof}
We remark that each $p_i$ is a polynomial in $z$ of degree $h(i)$ with coefficients in the ring $\bZ[1/2,\omega]$.
Next, we apply Melzak's Algorithm to the terminals of $T_k$ to obtain the locally minimum Steiner tree $S(T_k)$ with the same topology.
Surprisingly, it turns out that the Steiner points of $S(T_k)$ are also polynomials in $z$ with coefficients in $\bZ[1/2,\omega]$.

The first step in Melzak's algorithm is to calculate the so-called quasi-terminals $q_i$ ($1\leq i\leq 2^{k+1}-1$) \cite{BTW}.
For each $i=2^k,\dots,2^{k+1}-1$, let $q_i=p_i$.
Then, for each $i=2^k-1,\dots,1$, once $q_{2i}$ and $q_{2i+1}$ have been defined, let $q_i$ be the unique point such that the triangle $\upDelta_i=\triangle q_i q_{2i}q_{2i+1}$ is equilateral, and such that $p_i$ and $q_i$ are on opposite sides of the line $q_{2i}q_{2i+1}$.
Let $C_i$ be the circumcircle of $\upDelta_i$ and $c_i$ its centre (Fig.~\ref{circle-tri}).
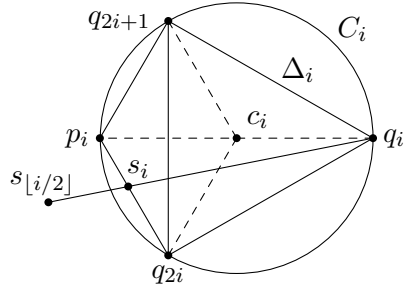
\begin{figure}
\centering
\begin{tikzpicture}[scale=0.6,line cap=round,line join=round,>=triangle 45,x=1.0cm,y=1.0cm]
\clip(-2.66,-1.94) rectangle (10.48,5.66);
\draw (1.46,4.48)-- (1.46,-0.7);
\draw (1.46,-0.7)-- (5.95,1.89);
\draw (5.95,1.89)-- (1.46,4.48);
\draw(2.96,1.89) circle (2.99cm);
\draw[dashed] (5.95,1.89)-- (-0.04,1.89);
\draw[dashed] (1.46,4.48)-- (2.96,1.89);
\draw[dashed] (1.46,-0.7)-- (2.96,1.89);
\draw (1.46,4.48)-- (-0.04,1.89);
\draw (1.46,-0.7)-- (-0.04,1.89);
\draw (5.95,1.89) -- (-1.1685180047812245,0.47320821327142154); 
\fill (1.46,4.48) circle (2.5pt) node[left=1mm] {$q_{2i+1}$};
\fill (1.46,-0.7) circle (2.5pt) node[below] {$q_{2i}$};
\draw (4.3,3.4) node {$\upDelta_i$};
\fill (5.95,1.89) circle (2.5pt) node[right] {$q_i$};
\draw (5.5,4.3) node {$C_i$};
\fill  (2.96,1.89) circle (2.5pt) node[above right] {$c_i$};
\fill  (-0.04,1.89) circle (2.5pt) node[left] {$p_i$};
\fill (-1.1685180047812245,0.47320821327142154) circle (2.5pt);
\draw (-1.3,0.9) node {$s_{\lfloor i/2\rfloor}$};
\fill (0.581447577711931,0.8216974325156448) circle (2.5pt);
\draw (0.8,1.2) node {$s_i$};
\end{tikzpicture}
\caption{Melzak's Algorithm}\label{circle-tri}
\end{figure}
Since $\myangle p_{2i}p_ip_{2i+1}=2\pi/3$, $\myangle p_{\lfloor i/2\rfloor}p_iq_i=\pi-\epsi$ and $\length{p_ip_{2i}}=\length{p_ip_{2i+1}}$, we obtain by induction that for $i=1,\dots,2^k-1$, $\myangle q_{2i}p_iq_{2i+1}=2\pi/3$, $\myangle q_ip_iq_{2i}=\myangle q_ip_iq_{2i+1}=\pi/3$, hence $p_i$ is on $C_i$ and the centre $c_i$ of $C_i$ is the midpoint of $p_i$ and $q_i$.
Also, $\length{p_iq_i}=2\length{p_iq_{2i}}=2\length{p_iq_{2i+1}}$.
Since $c_iq_{2i}p_iq_{2i+1}$ is a parallelogram, we have
\begin{align}
c_i &=p_i+(q_{2i}-p_i)+(q_{2i+1}-p_i) \notag\\
\intertext{and}
q_i &= p_i+2(c_i-p_i)\notag\\
&= p_i+2(q_{2i}-p_i)+2(q_{2i+1}-p_i).\label{5}
\end{align}
If $2^{k-1}\leq i< 2^k$, we have $q_{2i}=p_{2i}$ and $q_{2i+1}=p_{2i+1}$, hence
\begin{align*}
q_i &= p_i+2(p_{2i}-p_i)+2(p_{2i+1}-p_i)\\
&= p_i+(p_i-p_{\lfloor i/2\rfloor})\omega^{-1}z+(p_i-p_{\lfloor i/2\rfloor})\omega z\quad\text{by \eqref{*} and \eqref{***}}\\
&= p_i + (p_i-p_{\lfloor i/2\rfloor})z\\
&= p_i + \omega^{a_{k-1}(i)}\left(\frac{z}{2}\right)^{k-1}z\quad\text{by \eqref{4}},
\end{align*}
By induction, we obtain that for each $i<2^{k-1}$, (use \eqref{5}, \eqref{4}, \eqref{+}; see the Appendix)
\begin{equation}\label{qiformula}
q_i=p_i+\omega^{a_{h(i)}(i)}\left(\frac{z}{2}\right)^{h(i)}\sum_{j=1}^{k-h(i)}z^j 
\quad(i\geq 1).
\end{equation}
Therefore, each $q_i$ is a polynomial in $z$ of degree $k$.
In particular,
\begin{equation}\label{8}
q_1=\sum_{j=0}^k z^j. 
\end{equation}
Also, the centres
\begin{equation}\label{ciformula}
c_i=\frac12(p_i+q_i)=p_i+\frac12\omega^{a_{h(i)}(i)}\left(\frac{z}{2}\right)^{h(i)}\sum_{j=1}^{k-h(i)} z^j 
\end{equation}
are polynomials in $z$ of degree $k$.
In particular,
\begin{equation}\label{9}
c_1=1+\frac12\sum_{j=1}^k z^j. 
\end{equation}

Finally, we construct the Steiner points $s_i$, $1\leq i\leq 2^k-1$.
Formally, we let $s_0=p_0=0$.
Once $s_{\lfloor i/2\rfloor}$ has been constructed, $s_i$ is the point where the minor arc $\arc{q_{2i}q_{2i+1}}$ of $C_i$ intersects the segment $s_{\lfloor i/2\rfloor}q_i$.
See Figure~\ref{circle-tri}.
This gives the shortest Steiner tree for this tree topology as long as $\arc{q_{2i}q_{2i+1}}$ intersects $s_{\lfloor i/2\rfloor}q_i$.
This happens iff $\myangle s_{\lfloor i/2\rfloor}q_ip_i\leq\pi/6$ and $s_{\lfloor i/2\rfloor}$ is outside $C_i$.
For $i\geq 1$, we calculate $s_i$ by solving $\abs{s_i-c_i}=\abs{q_i-c_i}$, where
\begin{equation}\label{7}
s_i=q_i-\lambda(q_i-s_{\lfloor i/2\rfloor}),\quad 0<\lambda<1.
\end{equation}
If we square $\abs{q_i-\lambda(q_i-s_{\lfloor i/2\rfloor})-c_i}=\abs{q_i-c_i}$ and use conjugates, we can solve for $\lambda$:
\[ \lambda = \frac{q_i-c_i}{q_i-s_{\lfloor i/2\rfloor}} + \frac{\conj{q_i}-\conj{c_i}}{\conj{q_i}-\conj{s_{\lfloor i/2\rfloor}}},\]
and substitute into \eqref{7} to determine $s_i$:
\begin{align}
s_i &= q_i - \left(\frac{q_i-c_i}{q_i-s_{\lfloor i/2\rfloor}} + \frac{\conj{q_i}-\conj{c_i}}{\conj{q_i}-\conj{s_{\lfloor i/2\rfloor}}}\right)(q_i-s_{\lfloor i/2\rfloor})\notag\\
&= c_i - \frac{(\conj{q_i}-\conj{c_i})(q_i-s_{\lfloor i/2\rfloor})}{\conj{q_i}-\conj{s_{\lfloor i/2\rfloor}}}.\label{si-formula}
\end{align}
In particular, using \eqref{8} and \eqref{9}, $s_1=\frac12 + \frac12 z^k$.
It follows by induction (use \eqref{si-formula}, \eqref{ciformula}, \eqref{qiformula}; see the Appendix) that
\begin{equation}\label{siformula}
s_i = p_i + \frac{\omega^{a_{h(i)}(i)}}{2^{h(i)+1}}\left(\sum_{j=0}^{k-h(i)-1} z^j\right) 
(z^{h(i)+1}-1)\qquad(i=1,\dots,2^k-1).
\end{equation}
Next, we calculate the edge lengths of the Steiner tree.
\begin{align*}
s_{2i}-s_i &= p_{2i} + \frac{\omega^{a_{h(2i)}(2i)}}{2^{h(2i)+1}}\left(\sum_{j=0}^{k-h(2i)-1}z^j\right) 
(z^{h(2i)+1}-1)\\
 &\qquad - p_i - \frac{\omega^{a_{h(i)}(i)}}{2^{h(i)+1}}\left(\sum_{j=0}^{k-h(i)-1}z^j\right) 
 (z^{h(i)+1}-1)\quad\text{by \eqref{siformula}}\\
 &= \frac{\omega^{a_{h(2i)}(2i)}}{2^{h(2i)}} \Biggl[z^{h(2i)} + \frac{1}{2}\left(\sum_{j=0}^{k-h(2i)-1}z^j)\right) 
 (z^{h(2i)+1}-1)\\
 &\qquad\qquad\qquad-\omega\left(\sum_{j=0}^{k-h(2i)}z^j\right) 
 (z^{h(2i)}-1)\Biggr] \quad\text{by \eqref{+} and \eqref{4}.}
 \end{align*}
 Similarly,
\begin{align*}
s_{2i+1}-s_i &= \frac{\omega^{a_{h(2i+1)}(2i+1)}}{2^{h(2i+1)}} \Biggl[z^{h(2i+1)} + \frac{1}{2}\left(\sum_{j=0}^{k-h(2i+1)-1}z^j\right) 
(z^{h(2i+1)+1}-1)\\
 &\quad\qquad\qquad\qquad\qquad-\omega^{-1}\left(\sum_{j=0}^{k-h(2i+1)}z^j\right) 
 (z^{h(2i+1)}-1)\biggr].
 \end{align*}
 Let $h\in\set{1,\dots,k}$ and define 
 \[ p_{k,h}(z) = z^h + \frac12\left(\sum_{j=0}^{k-h-1}z^j\right) 
 (z^{h+1}-1)-\omega\left(\sum_{j=0}^{k-h}z^j\right) 
 (z^h-1)\]
 and
 \[ q_{k,h}(z) = z^h + \frac12\left(\sum_{j=0}^{k-h-1}z^j\right) 
 (z^{h+1}-1)-\omega^{-1}\left(\sum_{j=0}^{k-h}z^j\right) 
 (z^h-1).\]
 It follows that $s_{2i}-s_i=0$ iff $p_{k,h(2i)}(z) = 0$, and $s_{2i+1}-s_i=0$ iff $q_{k,h(2i+1)}(z) = 0$.
Since $p_{k,h}(1)=q_{k,h}(1)=1$, both $p_{k,h}(z)-1$ and $q_{k,h}(z)-1$ have $z-1$ as a factor.
In fact,
\begin{align*}
\abs{p_{k,h}(z)-1} &= \biggl\lvert \sum_{j=0}^{h-1}z^j 
+\frac12\sum_{j=0}^{k-h-1}z^j 
\sum_{j=0}^{h}z^j 
-\omega\sum_{j=0}^{k-h}z^j 
\sum_{j=0}^{h-1}z^j 
\biggr\rvert\cdot\abs{z-1}\\
&\leq \left(h+\frac12(k-h)(h+1)+(k-h+1)h\right)\abs{z-1}\\
&< k^2\abs{z-1},
\end{align*}
and similarly, $\abs{q_{k,h}(z)-1} < k^2\abs{z-1}$.
It follows that if $\abs{z-1}<1/k^2$, then $p_{k,h}(z)\neq0$ and $q_{k,h}(z)\neq0$.
Therefore, the Melzak construction gives a non-degenerate locally minimum Steiner tree for all $\epsi\in[0,1/k^2)$, since $\abs{z-1}\leq\epsi$.

The length of the Steiner tree is
\[L(S(T_k))=\length{p_0q_1}=\abs{\sum_{j=0}^k z^j}.\]
The modulus of this sum of complex numbers can be interpreted as the distance between the endpoints of a convex polygonal path consisting of $k+1$ segments of unit length with a turn of $\epsi$ between two adjacent segments.
This is easily calculated to be $\sin[(k+1)\epsi/2]/\sin(\epsi/2)$.
Thus, the ratio between the length of the approximate tree $T_k$ and the length of the locally minimum Steiner tree $S(T_k)$ is (recall \eqref{lengthTk})
\[ \frac{L(T_k)}{L(S(T_k))} = \frac{(k+1)\sin(\epsi/2)}{\sin[(k+1)\epsi/2]} \geq 1+\frac{k^2+2k}{24}\epsi^2.\]
Therefore, $G_2(\epsi,2^k+1) > (k\epsi)^2/24$ if $\epsi<1/k^2$, and Theorems~\ref{theorem:lowerbound} and \ref{theorem:lowerbound'} follow.
\end{proof}

\section{Conclusions}\label{section:conclusion}
\begin{enumerate}
\item In this paper we considered the planar case of the conjectures of Rubinstein, Wormald and Weng \cite{RWW}.
Although we proved one of their conjectures when $\epsi$ is sufficiently small in terms of the number of terminals (Corollary~\ref{cor:plane}), the full conjecture is still open even in the plane, a setting that one would have expected to be simple.
It is especially frustrating that for a small constant $\epsilon$ (for instance $\epsi=10^{-3}$), the best upper bound we have is $F_2(\epsi,n)=O(n)$ (Proposition~\ref{plane}).

\item In the $\epsi$-approximate Steiner tree constructed in Section~\ref{section:construction}, the edge lengths are halved at each new level of the tree.
If we let the edge lengths decay sufficiently fast, then most likely the topology of the $\epsi$-approximate tree will be the same as the topology of a minimum Steiner tree for $\epsi$ sufficiently small \cite{BRT}.
Thus, the locally minimum tree constructed using the Melzak algorithm as in Section~\ref{section:construction} will most likely be a minimum Steiner tree on the terminals.
This would then give a (miniscule) lower bound for $\overline{F}_2(\epsi,n)$.
However, the calculations are much harder when the ratio at which the edge lengths change are not exactly $1/2$, and we have not carried these out.
For similar ideas, see the papers \cite{BRT} and \cite{PST}.

\item In the proof of Theorems~\ref{theorem:lowerbound} and \ref{theorem:lowerbound'} (Section~\ref{section:construction}) we showed that the polynomials $p_{k,h}$ and $q_{k,h}$ do not have roots at distance smaller than $1/k^2$ from $1$.
We suspect that these polynomials actually have roots at distance approximately $c/k^2$ to $1$.

\item It is to be expected that the lower bound in Theorem~\ref{theorem:lowerbound'} should hold for general $n$, even if it turns out that $G_2(\epsi,n)$ is not monotone in $n$.
Most likely the proof can be adapted for values of $n$ other than $2^k+1$ by modifying the construction in Section~\ref{section:construction}, but we did not look at this in detail.

\item In the definitions of $F_d$, $\overline{F}_d$ and $G_d$ in Sections~\ref{section:problem} and \ref{section:results}, we could have included all $\epsi$-approximate trees on $n$ points instead of considering only the full ones.
However, by decomposing a Steiner tree into full components, it can be shown that the values of $F_d$, $d\geq 2$, and $G_d$, $d\geq 3$, will not change (use the inequality $\frac{a+b}{c+d}\leq\max\{\frac{a}{c},\frac{b}{d}\}$ and Propositions~\ref{monotone} and \ref{monotoneG}).
We do not know whether the values of $\overline{F}_d$ or $G_2$ will also be unchanged.
\end{enumerate}

\section*{Acknowledgements}
{\small
 Part of this paper was written while Swanepoel was visiting the Department of Mechanical Engineering, University of Melbourne, in March 2015.
Doreen Thomas is partially supported by the Melbourne Research Grant Support Scheme.
 The authors thank the referees for their helpful comments that lead to an improved paper.
 }

\appendix
\section*{Appendix. Induction steps}\label{AppendixA}
{\small
Here we provide the details of the induction proofs of \eqref{qiformula} and \eqref{siformula}.

First we assume that \eqref{qiformula} holds for $q_{2i}$ and $q_{2i+1}$:
\begin{align*}
q_{2i} &= p_{2i}+\omega^{a_{h(2i)}(2i)}\left(\frac{z}{2}\right)^{h(2i)}(z^1+z^2+\dots+z^{k-h(2i)})\\
q_{2i+1} &= p_{2i+1}+\omega^{a_{h(2i+1)}(2i+1)}\left(\frac{z}{2}\right)^{h(2i+1)}(z^1+z^2+\dots+z^{k-h(2i+1)})
\end{align*}
Then
\begin{align*}
q_i &= p_i + 2(q_{2i}-p_i) + 2(q_{2i+1}-p_i)\qquad\text{by \eqref{5}}\\
&= p_i + 2\left(p_{2i}-p_i + \omega^{a_{h(i)}(i)-1}\left(\frac{z}{2}\right)^{h(i)+1}\left(z^1+z^2+\dots+z^{k-h(i)-1}\right)\right)\\
&\phantom{\mathrel{=}} \phantom{p_i} + 2\left(p_{2i+1}-p_i + \omega^{a_{h(i)}(i)+1}\left(\frac{z}{2}\right)^{h(i)+1}\left(z^1+z^2+\dots+z^{k-h(i)-1}\right)\right)\quad\text{by \eqref{+}}\\
&= p_i + (p_i - p_{\lfloor i/2\rfloor})(\omega^{-1}+\omega)z + 2\omega^{a_{h(i)}(i)-1}\left(\frac{z}{2}\right)^{h(i)+1}(z^1+z^2+\dots+z^{k-h(i)-1})\\
&\mathrel{\phantom{=}} \phantom{p_i} + 2\omega^{a_{h(i)}(i)+1}\left(\frac{z}{2}\right)^{h(i)+1}(z^1+z^2+\dots+z^{k-h(i)-1})\quad\text{by \eqref{*} and \eqref{***}}\\
&= p_i + \omega^{a_{h(i)}(i)}\left(\frac{z}{2}\right)^{h(i)}z + \omega^{a_{h(i)}(i)}\left(\frac{z}{2}\right)^{h(i)}(z^2+z^3+\dots+z^{k-h(i)})\quad\text{by \eqref{4}}\\
&= p_i + \omega^{a_{h(i)}(i)}\left(\frac{z}{2}\right)^{h(i)}(z^1+z^2+\dots+z^{k-h(i)}),
\end{align*}
which is \eqref{qiformula}. \qed

Next, assume that
\[s_i = p_i + \frac{\omega^{a_{h(i)}(i)}}{2^{h(i)+1}} (1+z+\dots+z^{k-h(i)-1})(z^{h(i)+1}-1).\]
We have to show that
\begin{equation}\label{s2ind}
s_{2i} = p_{2i} + \frac{\omega^{a_{h(2i)}(2i)}}{2^{h(2i)+1}} (1+z+\dots+z^{k-h(2i)-1})(z^{h(2i)+1}-1)
\end{equation}
and
\begin{equation}\label{s2i+1nd}
s_{2i+1} = p_{2i+1} + \frac{\omega^{a_{h(2i+1)}(2i+1)}}{2^{h(2i+1)+1}} (1+z+\dots+z^{k-h(2i+1)-1})(z^{h(2i+1)+1}-1).
\end{equation}
By \eqref{si-formula},
\begin{equation}\label{s2i}
s_{2i} = c_{2i} - \frac{(\conj{q_{2i}}-\conj{c_{2i}})(q_{2i}-s_{i})}{\conj{q_{2i}}-\conj{s_{i}}}.
\end{equation}
By \eqref{ciformula},
\begin{equation}\label{c2i}
c_{2i} = p_{2i}+\frac12\omega^{a_{h(2i)}(2i)}\left(\frac{z}{2}\right)^{h(2i)}\left(z^1+z^2+\dots+z^{k-h(2i)}\right),
\end{equation}
and by \eqref{qiformula},
\begin{equation}\label{q2i}
\conj{q_{2i}} - \conj{c_{2i}} = \frac12\omega^{-a_{h(2i)}(2i)}(2z)^{-h(2i)}\left(z^{-1}+z^{-2}+\dots+z^{-k+h(2i)}\right).
\end{equation}
Next,
\begin{align*}
q_{2i}-s_i &= q_{2i}-p_i-\frac{\omega^{a_{h(i)}(i)}}{2^{h(i)+1}}\left(1+z+\dots+z^{k-h(i)-1}\right)(z^{h(i)+1}-1)\quad\text{by \eqref{si-formula}}\\
 &= q_{2i}-p_{2i}+p_{2i}-p_i-\frac{\omega^{a_{h(2i)}(2i)+1}}{2^{h(2i)}}\left(1+z+\dots+z^{k-h(2i)}\right)(z^{h(2i)}-1)\quad\text{by \eqref{+}}\\
&= \omega^{a_{h(2i)}(2i)}\left(\frac{z}{2}\right)^{h(2i)}\left(z+z^2+\dots+z^{k-h(2i)}\right) + \omega^{a_{h(2i)}(2i)}\left(\frac{z}{2}\right)^{h(2i)}\\
&\mathrel{\phantom{=}} - \frac{\omega^{a_{h(2i)}(2i)+1}}{2^{h(2i)}}\left(1+z+\dots+z^{k-h(2i)}\right)(z^{h(2i)}-1)\quad\text{by \eqref{qiformula} and \eqref{4}}\\
&= \left(\omega^{a_{h(2i)}(2i)}\left(\frac{z}{2}\right)^{h(2i)}(1-\omega) + \frac{\omega^{a_{h(2i)}(2i)+1}}{2^{h(2i)}}\right)\left(1+z+\dots+z^{k-h(2i)}\right)\\
&= \frac{\omega^{a_{h(2i)}(2i)}}{2^{h(2i)}}\left(\omega^{-1} z^{h(2i)} + \omega\right)\left(1+z+\dots+z^{k-h(2i)}\right),
\end{align*}
hence
\begin{align}
\frac{q_{2i}-s_{i}}{\conj{q_{2i}}-\conj{s_{i}}} &= \frac{\omega^{a_{h(2i)}(2i)}\left(\omega^{-1} z^{h(2i)} + \omega\right)\left(1+z+\dots+z^{k-h(2i)}\right)}{\omega^{-a_{h(2i)}(2i)}\left(\omega z^{-h(2i)} + \omega^{-1}\right)\left(1+z^{-1}+\dots+z^{-k+h(2i)}\right)}\notag\\
&= \omega^{2a_{h(2i)}(2i)}z^k.\label{q2i2}
\end{align}
If we substitute \eqref{c2i}, \eqref{q2i} and \eqref{q2i2} into \eqref{s2i}, we obtain \eqref{s2ind}.
The derivation of \eqref{s2i+1nd} is analogous. \qed
}

\end{document}